\documentclass[11pt]{amsart}
\usepackage{graphicx} 

\usepackage{hyperref}
\hypersetup{
  colorlinks   = true, 
  urlcolor     = blue, 
  linkcolor    = blue, 
  citecolor   = purple 
}
\usepackage[utf8]{inputenc}
\usepackage{fancyhdr}
\usepackage{amsmath}
\usepackage{amsthm}
\usepackage{amsfonts, amssymb}
\usepackage{amsthm}
\usepackage{mathrsfs}
\usepackage{mathtools}
\usepackage{graphics}
\usepackage{amssymb}
\usepackage{relsize}
\usepackage[all]{xy}
\usepackage{color}
\usepackage{tikz-cd}
\usepackage{tikz}
\usetikzlibrary{calc}
\usepackage{listings}
\usepackage{geometry}  
\usepackage{cite} 

\usepackage{array}

\usepackage[toc,page]{appendix}

\usepackage[most]{tcolorbox}

\usepackage{subcaption}
\usepackage{float}

\newcommand\OO{\mathcal{O}}

\newcommand\Aut{\text{Aut}}

\DeclareMathOperator{\sing}{Sing}
\DeclareMathOperator{\exc}{Exc}
\DeclareMathOperator{\supp}{Supp}

\DeclareMathOperator{\ord} {ord}

\DeclareMathOperator{\mult} {mult}
\DeclareMathOperator{\Diff} {Diff}
\DeclareMathOperator{\rk} {rk}

\DeclareMathOperator{\Spec} {Spec}

\newtheorem{theorem}{Theorem}[section]
\newtheorem{proposition}[theorem]{Proposition}
\newtheorem{lemma}[theorem]{Lemma}

\newtheorem{remark}[theorem]{Remark}
\newtheorem{example}[theorem]{Example}

\newtheorem{question}[theorem]{Question}

\newtheorem{definition}[theorem]{Definition}

\theoremstyle{definition}

\title{Recent progress on the Minimal Model Program for foliations}

\author{Paolo Cascini}

\author{Calum Spicer}

\date{}

\address{Department of Mathematics, Imperial College London, 180 Queen’s
Gate, London SW7 2AZ, UK}
\email{p.cascini@ic.ac.uk}

\address{Department of Mathematics, King's College London, Strand,
London WC2R 2LS, UK}
\email{calum.spicer@kcl.ac.uk}
\begin{document}

\begin{abstract}
We survey recent progress on the birational geometry of  foliations on
complex varieties.  We focus on the MMP viewpoint: singularities,  adjunction and
applications to the MMP for  foliations on surfaces and to the existence of flips
on threefolds.  
\end{abstract}

\maketitle

\tableofcontents

\section{Introduction}\label{s_introduction}

Foliations are useful tools in algebraic geometry and its neighbouring subjects:
they arise from fibrations, from algebraic group actions, and from differential
forms or vector fields.  From the birational point of
view one is led to study pairs $(X,\mathcal F)$ up to birational modification,
with the guiding principle that the canonical class of the foliation should
control the geometry of its leaves in much the same way that the canonical
divisor controls the geometry of a variety.

More concretely, if $\mathcal F$ is a foliation of rank $r$ on a normal variety
$X$, one can associate to it the canonical sheaf
$\omega_{\mathcal F} := (\wedge^r T_{\mathcal F})^{*}$ and a canonical divisor
$K_{\mathcal F}$ such that $\OO_X(K_{\mathcal F})\simeq \omega_{\mathcal F}$.
When the singularities of $\mathcal F$ are mild, a striking amount of the
technology of the Minimal Model Program (MMP) can be adapted to $K_{\mathcal F}$:
there are analogues of adjunction, cone and contraction theorems, and in
favourable situations one can run a $K_{\mathcal F}$-MMP.

At the same time, foliations introduce genuinely new phenomena.  Even on a
smooth ambient variety, there is no hope of resolving a singular foliation
to a smooth one by a sequence of blow-ups: blowing-up can create singularities
for the transformed foliation (cf.\ Example~\ref{ex_blow_up}), and the
appropriate goal is a \emph{reduction} of foliation singularities to a
distinguished class such as canonical singularities,
a problem which is currently understood in full generality only in low
dimensions.  Finally, the
base point free theorem is not available in the same generality, so the
construction of flips in dimension three and higher requires additional
structure (such as separatrices, meromorphic first integrals, or complements).

\medskip 

The aim of this survey is to give an accessible introduction to these ideas and
to highlight some of the recent progress in the birational geometry of
foliations, with an emphasis on the MMP perspective.  A recurring theme is
that, under appropriate singularity assumptions, $K_{\mathcal F}$ behaves
surprisingly like the canonical divisor of a mildly singular variety, while the
places where the analogy fails are often the source of interesting geometry.

\medskip

We pose two central questions which we hope motivate this perspective:

\begin{enumerate}
\item \emph{Bend and break for foliations.}
If there exists a curve $C \subset X$ with $K_{\mathcal F}\cdot C<0$, does
there exist a rational curve $\Sigma$ with $K_{\mathcal F}\cdot \Sigma<0$?
To what extent can $\Sigma$ be chosen to be ``oriented'' with respect to the
foliation (for instance tangent to $\mathcal F$, or contained in an invariant
subvariety)?

\item \emph{Existence of minimal models.}
Assume that $K_{\mathcal F}$ is pseudo-effective.  Does there exist a
$K_{\mathcal F}$-negative birational map $X\dashrightarrow X'$ such that for
the induced foliation $\mathcal F'$ on $X'$ the divisor $K_{\mathcal F'}$ is
nef?  More generally, under what hypotheses can one run a $K_{\mathcal F}$-MMP and what can be said about the output?
\end{enumerate}

\medskip

A key technical tool underlying many of the developments is adjunction for
foliations.  Given a divisor $D$ on $X$, one can restrict the foliation to the
normalisation of $D$ and obtain a foliation $\mathcal F_D$ together with an
effective correction term $\Diff_D(\mathcal F)\ge 0$ (the different)
satisfying a linear equivalence of the form
\[
(K_{\mathcal F}+\epsilon(D)D)\big|_D \;\sim\;
K_{\mathcal F_D}+\Diff_D(\mathcal F),
\]
see Theorem~\ref{t_adjunction}.  This formula plays for foliations much the same
role that adjunction plays for pairs in the usual MMP: it allows one to
translate numerical negativity on $X$ into concrete restrictions on 
subvarieties, and it provides a mechanism to compare singularities on $X$ and on
$D$ (cf.\ Theorem~\ref{thm_adj_sing}).  In particular, on surfaces it gives a
clean route to McQuillan's Minimal Model Program \cite{McQuillan08}, and in
dimension three it is one of the main inputs in the construction of foliated
flips.

\medskip

Our intention is not to be exhaustive, but rather to isolate a small collection
of ideas that have proved robust across several settings, and to
record a number of open problems that seem to us to be natural next targets.  

We also remark that there has recently been much substantial progress in understanding the birational geometry of foliations with algebraic leaves, and many new and exciting techniques have been developed in this setting. We will not say much about this here, but we refer to \cite{ACSS, LMX25, CS23a, chen2023mmp, m7a, m7b} to give an indication of the extent to which the field has been developing. 

It is also important to mention the large amount of work on classifying foliations of ``special type'', e.g., foliations such that $-K_{\mathcal F}$ is nef.  Again, we will not have the space to survey these developments here, but we refer the reader to \cite{MR3033631, MR3842065, Druel21, MR4910598}
for a small selection of the work done in this direction.

\medskip

\noindent\textbf{Organisation of the paper.}
In \S\ref{s_preliminaries} we recall the definitions and some basic results of some of the  objects used in the paper.
In \S\ref{s_singularities} we survey the singularities that naturally appear in the foliated MMP.
 In \S\ref{s_adjunction} we develop adjunction for foliations, and
explain how it controls the interaction between the singularities of the
foliation and the geometry of divisors.
In \S\ref{s_mmp2} we explain how these tools give a streamlined approach to the MMP for
 foliations over surfaces, recovering the main features of McQuillan's program and
some structural results in the case of general type, including the
appearance of elliptic Gorenstein leaves and the resulting failure of finite
generation in general.
Finally, in \S\ref{s_flips} we discuss some aspects of the MMP in higher dimensions, with a
focus on  flips on threefolds.  

\medskip

\noindent\textbf{Acknowledgements.} 
We would like to thank Andr\'{e} Belotto da Silva, Federico Bongiorno, St\'ephane Druel, Jihao Liu, James M\textsuperscript cKernan, Michael McQuillan, Jorge Pereira, and Roberto Svaldi for several 
discussions on the content of these notes. 
The first author is partially funded by a Simons Collaboration Grant and the second author is partially funded by EPSRC.
This survey was written for a volume arising from the Summer Research Institute in Algebraic Geometry held at Colorado State University in 2025, and we are grateful to the organisers for the invitation.
We would also like to thank the referee for carefully reading the paper
and for several useful suggestions and corrections.

\section{Preliminaries}\label{s_preliminaries}

We work over the field of complex numbers $\mathbb C$.

\subsection{Derivations and vector fields}\label{s_derivations}
Let $X$ be a normal variety. A {\bf derivation} of $X$ is a $\mathbb C$--linear map of sheaves
$\partial\colon \mathcal O_X \;\longrightarrow\;\mathcal O_X$
such that for every open set \(U\subset X\) and all \(f,g\in \mathcal O_X(U)\), the {\bf Leibniz rule} holds:
\[
\partial(fg) = \partial(f)g + f\partial(g).
\] 
We denote by $\Omega^1_X$ the sheaf of K\"ahler differentials of $X$ and, for any positive integer $r$, we let $\Omega^{[r]}_X:= (\Omega^r_X)^{**}$. We also define $T_X := (\Omega^1_X)^*$ to be the tangent sheaf of $X$.
A {\bf vector field} $V$ on $X$ is a  section of the tangent sheaf  and, in particular, it induces an $\mathcal O_X$-linear morphism
$V\colon \Omega^1_X\to \mathcal O_X$. Composing with the universal derivation
$d:\mathcal O_X\to \Omega^1_X$ we obtain a derivation
\[
\partial_V:\mathcal O_X \longrightarrow \mathcal O_X, \qquad
\partial_V(f):=V(df),
\]
since $d(fg)=f\,dg+g\,df$ and $V$ is $\mathcal O_X$-linear. Conversely, any derivation $\partial\colon \mathcal O_X\to \mathcal O_X$ uniquely factors through
$d$ by the universal property of $\Omega^1_X$, and therefore it  induces a section $V_\partial$
of $T_X$.

\medskip

Let $X$ be a normal variety and let $\partial$ be a derivation on $X$. The \textbf{singular locus} of $\partial$ is defined as 
\[
\sing \partial\ :=\ \bigl\{x\in X  \mid \partial(\mathcal O_{X, x}) \subset \mathfrak m_x \bigr\},
\]
 where $\mathfrak m_x \subset \mathcal O_{X, x}$ denotes the maximal ideal at $x$. Consider an embedding $X \subset \mathbb C^N$ and let $I_X$ be the ideal of $X$. A {\bf lift} of $\partial$ is a derivation $\tilde{\partial}$ on $\mathbb C^N$ such that $\tilde {\partial} (I_X)\subset I_X$ and 
 $\partial (f|_X)=\tilde{\partial}(f)|_X$ for any open set $U\subset \mathbb C^N$ and $f\in \mathcal O_{\mathbb C^N}(U)$. 
Then, for any $x\in X$, we have that $x\in \sing \partial$ if and only if $x\in \sing \tilde{\partial}$. 
 
Let
 $V_\partial$ be the
 vector field associated to $\partial$. If $X$ is smooth then $\sing \partial$ coincides with the zero locus of the section $V_\partial$, i.e.
\[
\sing \partial\ :=\ \bigl\{x\in X \mid (V_\partial)_x =0 \text{ in } (T_X)_x\otimes \mathbb C(x)\bigr\}.
\]
On the other hand, in general this equality does not hold as the following example shows: 
\begin{example}
Let $X=\{xy-z^2=0\}\subset \mathbb C^3$ and let  $x_0=(0,0,0)\in X$. Consider the derivation
\[
\widetilde\partial \ :=\ x\partial_x-y\partial_y\quad\text{on }\mathbb C^3.
\]
Since $\widetilde\partial(xy-z^2)\in(xy-z^2)$, it induces a derivation $\partial$ of $\mathcal O_X$ such that  $\partial(\mathcal O_{X,x_0})\subset \mathfrak m_{x_0}$. But,
 $\partial$ is \emph{not} zero in $T_X\otimes \mathbb C(x_0)$. Indeed, 
 one can verify that near $x_0$ the sheaf $T_X$ is generated by $x\partial_x-y\partial_y, ~ 2x\partial_x+z\partial_z$ and $~2y\partial_y+z\partial_z$ and that 
 $\dim_{\mathbb C}T_X\otimes \mathbb C(x_0) = 3$, hence $x\partial_x-y\partial_y$ is non-zero in $T_X\otimes \mathbb C(x_0)$.
\end{example}

\subsection{Foliations}\label{s_foliations}
Let $X$ be a normal   variety. 
A {\bf foliation } $\mathcal F$ of rank $r$ on $X$ is a rank $r$ coherent subsheaf $T_{\mathcal F}\subseteq T_X$
which is 
\begin{enumerate}
\item[(i)] closed under Lie bracket, i.e. the map 
$$\wedge^2 T_{\mathcal F}\to T_X/T_{\mathcal F}\qquad v\wedge w \mapsto [v,w]$$
is zero, and 
\item[(ii)] saturated, i.e., the {\bf normal sheaf} $\mathcal N_{\mathcal F}:= T_X/T_{\mathcal F}$ is torsion free.
\end{enumerate}
We denote by $\omega_{\mathcal F}:=(\wedge^r T_{\mathcal F})^{*}$ the {\bf canonical sheaf} of $\mathcal F$. A {\bf canonical divisor} of $\mathcal F$ is any Weil divisor $K_{\mathcal F}$ on $X$ such that $\mathcal O_X(K_{\mathcal F})\simeq \omega_{\mathcal F}$.

\medskip 

Below we collect some basic examples of foliations that will be useful
throughout the paper:

\begin{example}\label{example_first}
Let $X$ be a normal variety.

\begin{enumerate}
\item The equality $T_{\mathcal F}=T_X$ defines the trivial foliation, for which
$K_{\mathcal F}=K_X$.

\item Any vector field $V\in H^0(X,T_X)$ defines a rank one foliation on $X$.

\item Let $X$ and $Y$ be normal varieties. A {\bf fibration}
$f\colon X\to Y$ is a surjective morphism with connected fibres. It is easy
to check that $f$ induces a foliation defined by
\[
T_{\mathcal F}:=\ker\bigl(T_X\to f^*T_Y\bigr).
\]
Assume that $f$ is equidimensional and $Y$ is $\mathbb Q$-factorial. Then
\[
K_{\mathcal F}=K_{X/Y}+\sum (1-\ell_D)D,
\]
where the sum runs over all prime divisors $D$ on $X$ such that $f(D)$ is a
divisor on $Y$, and for any prime divisor $P$ on $Y$ we write
$f^*(P)=\sum \ell_D D$ (see, for example,
\cite[Notation~2.7 and \S2.9]{Druel15b}).

\item More generally, let $f\colon X\to Y$ be a fibration between normal
varieties and let $\mathcal G$ be a foliation on $Y$. We define a foliation
$\mathcal F$ on $X$ by
\[
T_{\mathcal F}:=\ker\bigl(T_X\to f^*T_Y\to f^*\mathcal N_{\mathcal G}\bigr).
\]
Then $\mathcal F$ has rank $\rk \mathcal F=\dim X -\dim Y+\rk \mathcal G$, and the
previous example corresponds to the case $T_{\mathcal G}=0$. We call
$\mathcal F$ the {\bf pullback foliation} of $\mathcal G$ along $f$ and we denote it by $f^{-1}\mathcal G$.
Even more generally, given a dominant rational map
$f\colon X\dashrightarrow Y$ and a foliation $\mathcal G$ on $Y$, one can
define the pullback foliation $\mathcal F=f^{-1}\mathcal G$ along $f$ on $X$. 
In particular,
any rational dominant map $f\colon X\dashrightarrow Y$ induces a foliation by
pulling back the foliation $T_{\mathcal G}=0$ on $Y$ along $f$. Any foliation arising
in this way is called {\bf algebraically integrable}.

\item Let $A=\mathbb C^2/\Lambda$ be an abelian surface, where
$\Lambda\subset \mathbb C^2$ is a lattice, and let $v\in\mathbb C^2$ be a
nonzero vector. Then $v$ defines a foliation
\[
\mathcal O_A=:T_{\mathcal F}\longrightarrow
T_A\simeq \mathcal O_A^{\oplus 2}.
\]
It is easy to check that this foliation is algebraically integrable if and only if
$\Lambda\cap \mathbb C v\simeq \mathbb Z^{\oplus 2}$.
\end{enumerate}
\end{example}

Assume for now that $X$ is a smooth variety and $\mathcal F$ is a foliation on $X$. The {\bf singular locus} of $\mathcal F$ is 
$$\sing \mathcal F:=\{ x\in X\mid \mathcal N_{\mathcal F}\text{ is not locally free at $x$ } \}.$$
Since by assumption $T_X/T_{\mathcal F}$ is torsion free,  it follows that $\sing \mathcal F$ has codimension at least two: this is similar to saying that the singular locus of any normal variety has codimension at least two. 

More importantly, if $\mathcal F$ is a foliation of rank $r$ on a normal variety $X$ then 
by {\bf Frobenius' theorem} (e.g. see \cite[Appendix \S 2]{CN85}), the fact that $T_{\mathcal F}$ is closed under Lie bracket implies that for any $x\in X\setminus (\sing X\cup \sing \mathcal F)$ there exist an analytic open neighbourhood $x\in U$ and a morphism $\phi\colon U\to \mathbb  C^q$ where $q=\dim X-r$, such that $\mathcal N_{\mathcal F}|_U=\phi^*T_{\mathbb C^q}$ and $T_{\mathcal F}|_U$ is the relative tangent bundle. 
A {\bf leaf} of $\mathcal F$ is an immersed analytic subvariety of $X$ which is locally a union of fibres of $\phi$.

\subsection{Pfaff fields}
Let $X$ be a normal variety and let $\mathcal F$ be a rank $r$ foliation
on $X$.  We can associate to $\mathcal F$ a morphism
\[\phi\colon \Omega^{[r]}_X \rightarrow \mathcal O_X(K_{\mathcal F})\]
defined by taking the double dual of the  $r$-wedge product of the map $\Omega^{[1]}_X\to T_{\mathcal F}^*$, induced by the inclusion
$T_{\mathcal F} \subset T_X$.  We will call $\phi$ the {\bf Pfaff field} (or {\bf Pfaffian map}) associated to $\mathcal F$.
Following \cite[Definition 5.4]{Druel21}, we define the {\bf twisted Pfaff field} 
as the induced map
\[\phi'\colon (\Omega^{[r]}_X \otimes \mathcal O_X(-K_{\mathcal F}))^{**} \rightarrow \mathcal O_X\]
and we define the {\bf singular locus} of $\mathcal F$,   denoted by $\sing \mathcal F$, to be the cosupport of the image of $\phi'$.
We say that $\mathcal F$ is {\bf smooth} at a closed point $x\in X$ if $x\notin \sing \mathcal F$ and we  say that $\mathcal F$ is a smooth foliation if $\sing \mathcal F$ is empty.
(We remark that what we call smooth is referred to as ``weakly regular'' in \cite{Druel21}.)

If $X$ is smooth, then the definition coincides with the definition in \S \ref{s_foliations}. Similarly, if $X$ is smooth and $\mathcal F$ is defined by a derivation $\partial$, 
then $\sing \mathcal F = \sing \partial$, as defined in \S \ref{s_derivations}. However, in general, it is unclear if these two notions of singularity coincide.  If $X$ is klt then this is proven in \cite[Proposition 2.32]{CS20}.
In general, we may ask:

\begin{question} 
    Do  these notions of singularity coincide on normal varieties?
\end{question}
Note that the above question is open even in the case of surfaces. 
Surprisingly, we have: 
\begin{example}\label{e_F=T}
 Let $X$ be a normal variety and assume  $T_{\mathcal F}=T_X$ then $\mathcal F$ is smooth, indeed
\[(\mathcal O_X(K_X) \otimes \mathcal O_X(-K_{X}))^{**} \simeq \mathcal  O_X\]
\end{example}

The following Lemma implies that the Pfaffian of a foliation uniquely determines the foliation: 

\begin{lemma}
\label{lem_linear_algebra}
Let $k$ be a field and let $V$ be a $k$-vector space of dimension $n$.  Fix an integer $0<r\leq n$ and let $L \subset \bigwedge^r V$ be a one 
dimensional subspace spanned by a non-zero vector $w$.
Consider the map $\alpha\colon V \to \bigwedge^{r+1} V$ defined by $v \mapsto v\wedge w$.
Then:
\begin{enumerate}
\item $\bigwedge^r \ker \alpha \subset L$.  In particular, if $\dim \ker \alpha = r$ then $\bigwedge^r \ker \alpha = L$.

\item If $W\subset V$ is a $r$-dimensional subspace such that $\wedge^r W=L$ then $W=\ker \alpha$.
\end{enumerate}
\end{lemma}
\begin{proof}
Note that if $v\in V$ is a non-zero vector such that $v\wedge w=0$, then there exists $w_1\in \bigwedge^{r-1} V$ such that $w= w_1\wedge v$. Thus, the Lemma follows easily. 
\end{proof}






\subsection{Invariant subvarieties}

Let $X$ be a normal variety and let $\partial$ be a derivation on $X$.
We say that an ideal $I \subset \mathcal O_X$ (equivalently, the subscheme $Y$ defined by $I=:I_Y$) is {\bf invariant} by $\partial$ provided $\partial(I) \subset I$.
More generally, if $\mathcal F$ is a foliation we say that $I$ (or $Y$ if $I=I_Y$) is {\bf $\mathcal F$-invariant} (or invariant by $\mathcal F$) if $\partial(I) \subset I$ for all local sections $\partial \in T_{\mathcal F}$.  
We remark that some care needs to be taken with this definition near the singular locus of $X$ and 
so, in some places (e.g. \cite{CS23b}), invariance of a variety by a foliation includes a local freeness condition on the canonical sheaf of the foliation in a neighbourhood of the generic point of $Y$. 

\begin{lemma}\label{l_invariant}
Let $X$ be a normal scheme, let $\mathcal{F}$ be a rank $r$ foliation on $X$ 
and let $W \subset X$ be a subvariety and assume that $T_{\mathcal F}$ is locally free in a neighbourhood of the generic point of $W$.

Then, in a neighbourhood of the generic point of $W$, the following are equivalent:
\begin{enumerate}
\item $W$ is $\mathcal F$-invariant;
\item there exists a decomposition
\begin{center}
\begin{tikzcd}
\Omega^r_X\vert_W \arrow[r, "\phi"] \arrow[d] & \mathcal O_W(K_{\mathcal F})\\
\Omega^r_W \arrow[ur] & 
\end{tikzcd};
\end{center}
and
\item the map $T_{\mathcal F}|_W\to {T_X}|_W$ factors through $T_W$. 
\end{enumerate}
\end{lemma}

\begin{proof}
Let $I_W$ be the ideal of $W$. 
Everything is local about the generic point of $W$, so we may freely replace $X$ 
by a neighbourhood of the generic point of $W$ and therefore we may assume that 
$T_{\mathcal{F}}$ is locally free and that  $\partial_1, \ldots, \partial_r$ 
are generators of $T_{\mathcal{F}}$.

We prove that (1) implies (2).  Since $W$ is $\mathcal F$-invariant, we have 
 that $\partial_i(I_W) \subset I_W$ for all $i=1,\dots,r$. Thus, 
$(\partial_1 \wedge \cdots \wedge \partial_r)(df \wedge \beta)$
vanishes along $W$ for any $f \in I_W$ and any $(r-1)$-form $\beta$. 
In particular,
\[
\ker \big(\Omega_X^r\big|_W \to \Omega_W^r\big) =
(dI_W \wedge \Omega_X^{r-1})\big|_W  
\subset \ker \big(\Omega_X^r\big|_W \to \mathcal{O}_W(K_{\mathcal{F}})\big).
\]
This implies that in a neighbourhood of the generic point of $W$, the morphism 
$\Omega_X^r\big|_W \to \mathcal{O}_W(K_{\mathcal{F}})$
factors through $\Omega_X^r\big|_W \to \Omega_W^r$, as claimed.

\medskip 

The rest of the proof is easy and it is left as an exercise. 
\end{proof}

\begin{definition}
    Let $\mathcal F$ be a foliation on a normal variety $X$ and let $W\subset X$ be a prime divisor. We introduce the indicator function $\epsilon(W)$
by defining 
\[
\epsilon(W) :=
\begin{cases} 0 ~ \text{if $W$ is $\mathcal F$-invariant} \\ 1 \text{ if $W$ is not $\mathcal F$-invariant.}
\end{cases}
\]
\end{definition}



\subsection{Resolution vs. reduction}

The first thing to note about foliation singularities is that it is not possible to resolve a singular foliation to a smooth one by a sequence of blow-ups.  Indeed, it is worth comparing the problem 
of resolving the singularities of a morphism with resolving the singularities of a foliation.  It is not possible to ``resolve'' a singular morphism to a smooth one by blow-ups, rather the best one can hope to achieve is toroidalisation (or monomialisation) of the morphism by blow-ups (cf. Theorem \ref{t_AK}).  Similarly, the best one can hope for is to achieve the reduction (as opposed to resolution) of singularities
of the foliation to some distinguished class of foliation singularities.
However, it is unclear what exactly this distinguished class of singularities should be in general, and the current best results on reduction of singularities are only known in dimension at most three. We refer to \cite{MR220710, CC92, cano04, Panazzolo, MP13, ABTW25} for some of the state of the art along these lines.

\subsection{Blowing-up vector fields}
As the following example shows, blowing-up smooth foliation at a point will typically make the transformed foliation singular:

\begin{example}
\label{ex_blow_up}
Let $X=\mathbb C^n$ and let $q=(0,\dots,0)$. Consider the blow-up $p\colon Y\to X$ of $X$ at $q$. In a chart, we may write
$$p(u_1,\dots,u_n)=(u_1,u_1\cdot u_2,\dots,u_1\cdot u_n)=(x_1,\dots,x_n).$$
Thus, on the open set $\{x_1\neq 0\}$, we may write 
$$p^{-1}(x_1,\dots,x_n)
=\Bigl(x_1,\;\tfrac{x_2}{x_1},\;\dots,\;\tfrac{x_n}{x_1}\Bigr)$$
and the Jacobian of $p^{-1}$ is given by 
$$J\bigl(p^{-1}\bigr)
=
\begin{pmatrix}
1 & -\dfrac{x_2}{x_1^2} & \cdots & -\dfrac{x_n}{x_1^2} \\[1ex]
0 & \dfrac{1}{x_1}       &        & 0                 \\[1ex]
\vdots &                   & \ddots & \vdots           \\[1ex]
0 & 0                   &        & \dfrac{1}{x_1}
\end{pmatrix}$$
Thus, 
$$J(p^{-1})|_{p(u_1,\dots,u_n)}=
\begin{pmatrix}
1 & -\dfrac{u_2}{u_1} & \cdots & -\dfrac{u_n}{u_1} \\[1ex]
0 & \dfrac{1}{u_1}       &        & 0                   \\[1ex]
\vdots &                   & \ddots & \vdots             \\[1ex]
0 & 0                   &        & \dfrac{1}{u_1}
\end{pmatrix}.
$$
Since $p^*\partial = \partial \circ J(p^{-1})$, 
it follows that 
\begin{align*}
p^*\!\bigl(\partial_{x_1}\bigr)
&=
\partial_{u_1}
-\frac{1}{u_1}\bigl(u_2\partial_{u_2}+\cdots+u_n\partial_{u_n}\bigr),\\[1ex]
p^*\!\bigl(\partial_{x_i}\bigr)
&=
\frac{1}{u_1}\,\partial_{u_i}
\quad \text{for }i=2,\dots,n.
\end{align*}
Note that if  $\partial = \partial_{x_1}$ then $\partial$ is smooth and 
$$p^*(\partial_{x_1}\bigr)
=\frac 1 {u_1} \partial_Y$$
where 
$$\partial_Y=u_1 \partial_{u_1} - (u_2\partial_{u_2}+\cdots+u_n\partial_{u_n}).$$
In particular, $\partial_Y$ is not smooth at $(0,\dots,0)$. 
\end{example}

\subsection{Riemann--Hurwitz formula}

\begin{proposition}[Riemann--Hurwitz formula \cite{Druel21}]
\label{p_RH}
Let $\sigma \colon Y \to X$ be a finite surjective morphism between normal varieties, 
let $\mathcal{F}$ be a foliation on $X$ such that $K_{\mathcal F}$ is $\mathbb Q$-Cartier and let $\mathcal{G} := \sigma^{-1}\mathcal F$ be the pullback foliation (cf. Example \ref{example_first}).

Then
\begin{enumerate}
\item for any prime divisor $D$ on $Y$ we have $\epsilon(\sigma(D)) = \epsilon(D)$; and 
    \item we may write
\[
K_{\mathcal{G}} \;=\; \sigma^* K_{\mathcal{F}} \;+\; \sum \epsilon(\sigma(D))(r_D - 1)D
\]
\end{enumerate}
where the sum runs over all the prime divisors on $Y$ and $r_D$ is the ramification index of $\sigma$ along $D$.  
In particular, if every ramified divisor is $\mathcal{G}$-invariant then 
\[
K_{\mathcal{G}} = \sigma^* K_{\mathcal{F}}.
\]
\end{proposition}

\begin{proof}
Let $q=\dim X - \rk \mathcal F$. By the classic Riemann-Hurwitz theorem, to prove (2) it is enough to prove that
$$\wedge^q\mathcal N^*_{{\mathcal F}_Y}\simeq \sigma^*(\wedge^q \mathcal N^*_{{\mathcal F}}) \otimes \mathcal O_Y\left (\sum (1-\epsilon(\sigma (D)))(r_D - 1)D\right). $$

We just need to check the formula around the general point $\eta$ of a prime divisor $D$ on $Y$. Thus, we may assume that  $Y, \mathcal F_Y$ and $D$ are smooth at $\eta$  and that $X,\mathcal F$ and $\sigma(D)$ are smooth at $x:=\sigma(\eta)$. 
We can also assume that there are coordinates $(y_1,\dots,y_n)$ of $Y$ around $\eta$ such that 
$$\sigma(y_1,\dots,y_n)=(y_1^m,y_2,\dots,y_m)\qquad\text{and}\qquad D=\{y_1=0\}.$$

Suppose first that $\sigma(D)$ is $\mathcal F$-invariant. Then $\mathcal F$ is defined by a $q$-form 
\[
dx_1 \wedge \alpha_1 + x_1 \alpha_2,
\]
where $\alpha_1 \in \wedge^{q-1}(\mathcal{O}_{X,x}\, dx_2 \oplus \cdots \oplus \mathcal{O}_{X,x}\, dx_n)$ is nowhere vanishing and 
$\alpha_2 \in \Omega^q_X$. 
It follows that
\[
d\sigma (dx_1 \wedge \alpha_1 + x_1 \alpha_2) 
= y_1^{m-1} \bigl(m\, dy_1 \wedge d\sigma(\alpha_1) + y_1\, d\sigma (\alpha_2)\bigr)
\]
vanishes at order $m-1$ along $D$ and our claim follows. 

If $\sigma(D)$ is not $\mathcal{F}$-invariant, then $\mathcal{F}$ is given by a nowhere vanishing $q$-form 
$\alpha_3 \in \wedge^q(\mathcal{O}_{X,x}\, dx_2 \oplus \cdots \oplus \mathcal{O}_{X,x}\, dx_n)$. 
But then $d\sigma (\alpha_3)$ is a nowhere vanishing $q$-form and our claim follows.   We finally observe that (1) is an easy consequence of the local description  given above.
\end{proof}

\section{Singularities of foliations}
\label{s_singularities}

\subsection{Singularities of the MMP}

We start off by defining the discrepancy of a foliation:



\begin{definition} Let $\mathcal F$ be a foliation on a normal variety $X$ and let $\Delta\ge 0$ be an $\mathbb R$-divisor 
with coefficients in $(0,1]$
such that $K_{\mathcal F}+\Delta$ is $\mathbb R$-Cartier and let $\varphi\colon Y\to X$ be a birational morphism. 
Then we may write
$$K_{\mathcal F_Y}- \sum_E a(E,\mathcal F,\Delta)E=\varphi^*(K_{\mathcal F}+\Delta) $$
where 
$\mathcal F_Y=\varphi^{-1}\mathcal F$ is the pullback foliation  on $Y$, 
and the sum runs over all the prime divisors $E$ on $Y$.  The rational number $a(E,\mathcal F,\Delta)$ is called the {\bf discrepancy} of $(\mathcal F,\Delta)$ with respect to $E$. Note that it only depends on the valuation associated to $E$ and not on the choice of birational morphism $\varphi \colon Y\to X$. 

We say that  $(\mathcal F,\Delta)$ is {\bf terminal} (resp. {\bf canonical, log canonical}) if for all divisorial valuations $E$ (i.e. for any birational morphism $\varphi\colon Y\to X$) we have
$$a(E,\mathcal F,\Delta)> 0\qquad \text{resp. }  \ge 0, \quad \ge -\epsilon(E).$$
If $\Delta=0$, we will then just omit $\Delta$ in the notation. If $\partial$ is a derivation on $X$ such that $\sing\partial$ has codimension at least two, then we denote $a(E,\partial):=a(E,\mathcal F)$ where $\mathcal F$ is the foliation generated by $\partial$.
\end{definition}






Note that in \cite{MP13} what is called ``discrepancy'' is what we call ``log discrepancy'', i.e., $a(E, \mathcal F, \Delta)+\epsilon(E)$.


\begin{example}
Let $X=\mathbb C^n$ and let $q=(0,\dots,0)$. Let $p\colon Y\to X$  be 
 the blow-up of $X$ at $q$.  Let $E=\exc p$. Then the same calculations as in Example \ref{ex_blow_up} imply that $a(E,\partial_{x_1})=1$. 
 \end{example}

One advantage of the classes of canonical and log canonical singularities is that they make sense on singular varieties. 
The other benefit is that they are preserved (almost by definition) by the kind of operations we do in the MMP.

We also emphasise that these definitions do not depend on the existence of a resolution of singularities, although without a resolution statement it can be difficult in practice to determine 
whether or not a given singularity is log canonical.
However, as a sanity check we do have the following:

\begin{proposition}
Let $X$ be a smooth variety and let $\mathcal F$ be a smooth foliation on $X$. Then, $\mathcal F$ has canonical singularities.
\end{proposition}
\begin{proof}
This follows from \cite[Lemma 5.9]{Druel21}.
\end{proof}

\begin{remark}
We remark that a smooth foliation is not necessarily terminal: indeed, consider as an example the foliation $\mathcal F$ induced by a smooth morphism $f\colon X\to Y$
from a smooth threefold $X$ to a smooth surface $Y$. Then, it is an easy exercise to show that $\mathcal F$ is smooth but 
the blow-up in a fibre will always extract a divisor of discrepancy zero with respect to $\mathcal F$.
\end{remark}

While the inclusion of $\epsilon(E)$ in the definition of log canonicity is important philosophically (and from the perspective of adjunction), it is in fact equivalent to requiring 
that the discrepancies are all at least $ -1$, as the following Lemma shows:

\begin{lemma}
    Let $X$ be a normal variety, let $\mathcal F$ be a foliation on $X$ and let $\Delta \ge 0$ be an $\mathbb R$-divisor such that $K_{\mathcal F}+\Delta$ is $\mathbb R$-Cartier. 
    
        Then $(\mathcal F, \Delta)$ is log canonical if and only if  $a(E, \mathcal F, \Delta) \ge -1$ for all divisors $E$ over $X$.
\end{lemma}
\begin{proof}
One direction is clear. Suppose that $a(E, \mathcal F, \Delta) \ge -1$ for all divisors $E$ over $X$. We want to show that $(\mathcal F,\Delta)$ is log canonical. 
It suffices to show that if $E$ is any $\mathcal F$-invariant divisor over $X$ then in fact $a(E, \mathcal F, \Delta) \ge 0$.

So, suppose for the sake of contradiction that there exists a birational morphism $\pi\colon X' \to X$ which extracts an invariant divisor $E$ with $a(E, \mathcal F, \Delta)<0$.  
Arguing along the lines of \cite[Remark 2.3]{CS21} (i.e., repeatedly blowing-up some divisor in $E$)
we see that there exists some higher birational model $X'' \to X$ which extracts a divisor $E'$ of discrepancy less than $-1$, a contradiction.    
\end{proof}

We say that $\mathcal F$ is a (log) terminal (resp. (log) canonical) at a (not necessarily closed) point $x \in X$ if the obvious inequalities on divisors hold for any divisor $E$ whose centre on $X$ is the closure of
$x$.

\begin{example}
Let us continue to use notation as in Example \ref{ex_blow_up}.

Consider now the case $n=2$ and $p(u,v)=(u,uv)$ with exceptional divisor $E$. If $\partial=x\partial_x+\lambda y\partial_y$, with $\lambda\in \mathbb C$ then 
$$p^*\partial = u\partial_u+(\lambda -1)v \partial_v.$$
It follows that $a=a(E,\partial)=0$ unless $\lambda =1$.  It is easy then to check that if $\lambda \in \mathbb Q_{>0}$ then $\partial$ is not canonical. 
\end{example}

Let $X$ be a normal variety, let $\partial$ be a derivation on $X$ and let $(\mathcal F,D)$ be the {\bf associated foliated pair}, i.e. $\mathcal F$ is the rank one foliation generated by $\partial$
and  $D$ is the sum of the codimension one components of $\sing \partial$, taken with multiplicity. 
We say that $\partial$ is {\bf log canonical} if $(\mathcal F, D)$ is log canonical.

\begin{remark} Note that if $\partial$ is a derivation on a normal variety $X$, $(\mathcal F, D)$ is the associated foliated pair and $b\colon X' \to X$ is a birational morphism which extracts a divisor $E$
then we can interpret the discrepancy $a(E, \mathcal F, D)$ in terms of $\partial$ as follows.  There is a lift of $\partial$ to a meromorphic vector field $\partial'$ on $X'$.
In a neighbourhood of a general point $x \in E$ let $\{t = 0\}$ be a local equation for $E$ and write $\partial' = t^{-a}\delta$ where $\delta$ is a holomorphic vector field and $a$ is an integer. 
We then have an equality $a = a(E, \mathcal F, D)$.
\end{remark}

\subsection{Singularities of rank one foliations}

In this section, we review some features of the singularities of rank one foliations using the framework of the MMP described in the previous section.

Let us note that if $x \in X$ is a germ of a normal variety and $\mathcal F$ is a rank one foliation on $X$ such that $K_{\mathcal F}$ is Cartier, then locally $\mathcal F$ can be
defined by a single vector field $V$ or, equivalently, by the associated derivation $\partial_V$ (cf. Section \ref{s_derivations}). Indeed, in this case $T_{\mathcal F}$ is locally free of rank one, hence locally generated by a single section.

\subsubsection{Jordan decomposition of vector fields} 
Let $x\in X$ be a point of a normal variety and let $\partial$ be a vector
field such that $x\in \sing \partial$. Then 
by the Leibniz rule, we have that $\partial(\mathfrak m^n_x)\subset \mathfrak m^n_x$, for any positive integer $n$. Thus,
$\partial$ induces, for every
$n\ge 2$, a linear endomorphism
\[
\partial_n \colon \mathfrak m_x/\mathfrak m_x^n \to \mathfrak m_x/\mathfrak m_x^n .
\]
Following \cite{Mar81}, taking the Jordan decomposition of these 
endomorphisms and passing to the inverse limit, we obtain a canonical
decomposition
\[
\widehat\partial = \partial_S + \partial_N
\]
of the induced derivation $\widehat\partial$ of the completed local ring
$\widehat{\mathcal O}_{X,x}$, where $\partial_S$ is semisimple and
$\partial_N$ is nilpotent.

In general, this decomposition is not induced by holomorphic (or
algebraic) vector fields on a neighbourhood of $x$, but it exists only at the
level of the formal completion. In suitable formal coordinates, we may write
\[
\partial_S = \sum \lambda_i x_i \partial_{x_i}
\]
for some  $\lambda_i\in \mathbb C$.

\begin{lemma}\label{l_jordan}Notations as above.

Then $\partial_S(\widehat{\mathcal O}_{X,x}) \subset \widehat\partial(\widehat{\mathcal O}_{X,x})$.
In particular, 
$\sing(\widehat\partial) \subset \sing(\partial_S)$.
\end{lemma}

\begin{proof}  Let
$\widehat R := \widehat{\mathcal O}_{X,x}$ and let $\mathfrak m \subset \widehat R$ be the maximal ideal.
It suffices to show that for all integers $n \ge 0$ we have $\partial_{S, n}(\mathfrak m/\mathfrak m^n) \subset \partial_n(\mathfrak m/\mathfrak m^n)$, where $\partial_n$ (resp. $\partial_{S, n}$) is the linear map $\mathfrak m/\mathfrak m^n \to \mathfrak m/\mathfrak m^n$ induced by $\partial$ (resp. $\partial_S$).

By the properties of the Jordan decompositions, there exists a polynomial $p(t) \in \mathbb C[t]$ without constant term such that $\partial_{S, n} = p(\partial_n)$, considered as endomorphisms of the underlying
$\mathbb C$-vector space of $\widehat R/\mathfrak m^n$.
If we write $p(t) = tq(t)$ we see that 
$\partial_{S, n} = \partial_n \circ q(\partial_n)$, which implies that
$\partial_{S, n}(\widehat R/\mathfrak m^n) \subset \partial_n(\widehat R/\mathfrak m^n)$, as required.

The inclusion of singular loci follows immediately.
\end{proof}

\subsubsection{Log canonical vector fields} 
The following fundamental observation is due to McQuillan and Panazzolo, and provides a characterisation of singular vector fields.  The proof presented here arose out of joint discussions with Andr\'e Belotto da Silva.

\begin{proposition}
\label{prop_lc_iff_nonnilp}
Let $x \in X$ be a germ of a normal variety and let $\partial$ be a vector field such that $x \in \sing \partial$.  

\begin{enumerate}
\item If $\partial_S = 0$ (equivalently, the linear part of $\partial$ is nilpotent), then there is a single weighted blow-up which extracts a divisor 
of discrepancy less than $-1$ and, in particular, $\partial$ is not log canonical.

\item If $\partial$ is not log canonical, then $\partial_S  = 0$  (equivalently, the linear part of $\partial$ is nilpotent).
\end{enumerate}

In other words, a vector field is log canonical if and only if its linear part is non-nilpotent.

\end{proposition}

\begin{proof}
This is proven in \cite[Fact I.ii.4]{MP13}, the only difference is that (1) is a bit more precise than the claim in \cite{MP13}.  



We first prove (1).
Up to a change of coordinates, we may assume that the linear part of $\partial$ is in Jordan normal form.

Let $\pi\colon X' \to X$ be the weighted blow-up with weights $w=(w_1,\dots,w_n)$ and let $U'\subset X'$ be a chart so that there exists a  cyclic cover
$\sigma\colon \widetilde U \to U'$ and coordinates $(t,y_2,\dots,y_n)$ on $\widetilde U$ such that the induced morphism $b\colon\widetilde U\to X$ is given by 
$$b(t,y_2,\dots,y_n)=(t^{w_1},t^{w_2}y_2,\dots,t^{w_n}y_n).$$
Since we only need to estimate the vanishing order of $\pi^*\partial$ along the exceptional divisor $E$ of $\pi$, 
we claim that it is enough to show that 
$t^r$ divides $b^*\partial$ where $r$ is an integer greater than $1$.
Indeed, let $s$ be the ramification index of the cover $\widetilde U \to U'$ along $E$, let $e$ be the order of vanishing of $\pi^*\partial$ along $E$ and let $\tilde e$ be the order of vanishing of $b^*\partial$ along $\{t = 0\} = \sigma^{-1}(E)$.
The Riemann-Hurwitz formula
(cf. Proposition \ref{p_RH})
tells us that $\tilde e = se-\epsilon(E)(s-1)$, hence if $\tilde e > \epsilon(E)$, then $e >\epsilon(E)$.

Similarly to Example \ref{ex_blow_up}, we can compute 
\[b^*\partial_{x_1} = \frac{1}{w_1t^{w_1-1}}\partial_t - \sum_{i = 2}^n \frac{w_iy_i}{w_1t^{w_1}}\partial_{y_i}\]
and \[b^*\partial_{x_i} = \frac{1}{t^{w_i}}\partial_{y_i}\qquad \text{for $i=2,\dots,n$}.\]

Write $\partial = \partial_1+\partial_2$ where $\partial_1$ is the linear part of $\partial$ and $\partial_2$ is the part of order
at least $2$.
Since the linear part is in Jordan form we can write it as $\sum_{i = 1}^{n-1} \varepsilon_ix_i\partial_{x_{i+1}}$ where $\varepsilon_i \in \{0, 1\}$ for $i=1,\dots,n-1$. 
Our above calculations show that 
\[b^*(x_i\partial_{x_{i+1}})= t^{w_i - w_{i+1}}y_i\partial_{y_{i+1}},\]
where we define $y_1 = 1$,
 so if we set
$w_1\ge w_2+r \ge w_3+2r \ge \dots \ge w_{n}+(n-1)r$ then $t^r$ divides $b^*\partial_1$.  
Turning to $\partial_2$, we see that if $k \ge 2$, then \[b^*(x_ix_j\partial_{x_k}) = t^{w_i+w_j-w_k}y_iy_j\partial_{y_k},\] where again we define $y_1 = 1$.
Similarly, we see that 
\[
b^*(x_ix_j\partial_{x_1}) = t^{w_i+w_j-w_1}y_iy_j\left( \frac{t}{w_1}\partial_t - \sum_{i = 2}^n \frac{w_iy_i}{w_1}\partial_{y_i}\right ).
\]
So, if we assume that $w_1 \ge w_2 \ge \dots \ge w_n$ and $2w_n \ge w_1+r$, so that $w_i+w_j-w_k \ge 2w_n-w_1 \ge r$, then it follows from the above calculations that $t^r$ divides $b^*(x_ix_j\partial_{x_k})$, and hence if ${\bf m}$ is any multi-index with $|{\bf m}| \ge 2$ then $t^r$ divides $b^*(x^{\bf m}\partial_{x_k})$. We then deduce that, with this choice of weights, we have that  $t^r$ divides $b^*\partial_2$.

Thus, if we take $w_1 = w_2+r = \dots = w_n+(n-1)r$ and  $w_n \ge nr$, so that $w_1+r =  w_n+nr \leq 2w_n$, 
then the weighted blow-up with these weights will extract a divisor of discrepancy at most $-r$. Hence, (1) follows. 

\medskip

We now prove (2).
  We will suppose that $\partial$ is not log canonical and will show that the linear part of $\partial$ is nilpotent.  Let $\pi\colon X' \to X$ be a birational morphism and let $E$ be a prime divisor on $X'$ so that  $a(E, \mathcal F, D) < -\epsilon(E)$ where $(\mathcal F, D)$ is the foliated pair associated to $\partial$.   Note that around a general point of $E$ we may write $\pi^*\partial = t^b\delta$ where $\{t = 0\}$ is a local equation of $E$, $b >0$ and 
$\delta$ is a local vector field which leaves $E$ invariant.  In particular, it follows that if  $\mathfrak m:= \mathfrak m_x \subset \mathcal O_{X, x}$ is the maximal ideal then for any $f \in \mathfrak m$ we have 
\begin{align}\label{ord_ineq}\ord_E(f) < \ord_E(\partial(f)).\end{align}
Suppose for the sake of contradiction that 
the linear part $\partial_0$ of $\partial$ is not nilpotent. 
Then there exist a nonzero element $\bar f \in \mathfrak m/\mathfrak m^2$ and a scalar $\mu \in \mathbb C^*$ such that $
\partial_0\bar f = \mu \bar f$. 
Let $f \in \mathfrak m$ be any lift of $\bar f$. Then
\begin{align}\label{ord_ineq2}
\partial(f) \equiv \mu f \mod{\mathfrak m^2}.
\end{align}
Replacing $f$ by $f+\phi$ for some suitable choice of $\phi \in \mathfrak m^2$, we may freely assume that 
\[\ord_E(f) = \sup \{\ord_E(g) ~ | ~ g \equiv\lambda f \mod \mathfrak m^2, \lambda \in \mathbb C^*\} < +\infty\]
where the last inequality holds by \cite[Lemma 7]{MR9603}- indeed, if $I_n  = \{f \in \mathcal O_{X, x} : \ord_E(f) \ge n\}$ then \cite[Lemma 7]{MR9603} implies that for some $n \gg 0$ we have $I_n \subset \mathfrak m^2$.
On one hand, by (\ref{ord_ineq2})  and by our choice of $f$ it follows that $\ord_E(f) \ge \ord_E(\partial(f))$.
On the other hand, by (\ref{ord_ineq}) we have $\ord_E(f) < \ord_E(\partial(f))$.  This is our sought after contradiction.
\end{proof}


\begin{remark}
Let us also remark that this characterisation of log canonicity for vector fields shows that if $\mathcal F$ is a rank one foliation on a normal variety $X$, then the locus of 
points where $\mathcal F$ is log canonical is an open set (cf. \cite[Definition-Summary I.ii.6]{MP13} or \cite[Proposition 2.6]{Panazzolo}).  We remark that in general it is unclear if the log canonical locus of a general 
foliation is open.  This is true for foliations on threefolds and algebraically integrable foliations (including the case of foliations with one leaf, i.e. varieties) thanks to the relevant reduction of singularities statements, but seems to be unknown in general. 
\end{remark}

\begin{question}
Let $X$ be a normal variety and let $\mathcal F$ be a foliation on $X$.  Is it the case that the locus of log canonical singularities is Zariski open?
\end{question}

Let us also remark that Proposition \ref{prop_lc_iff_nonnilp} is remarkable in the sense that it gives an effective linear criterion for verifying log canonicity.  
It would be very interesting to know if there is some comparable criterion for foliations of higher rank:

\begin{question}
Is there a generalisation of Proposition \ref{prop_lc_iff_nonnilp} for a foliation $\mathcal F$ on a normal variety $X$ such that $T_{\mathcal F}$ is locally free?
\end{question}

For recent progress in this direction, see \cite{bongiorno2026}. As a consequence of Proposition \ref{prop_lc_iff_nonnilp}, we have: 

\begin{lemma}
\label{lem_codim_one_zero_ss}
Let $x \in X$ be a germ of a normal variety and let $\partial$ be a log canonical vector field and suppose that $\sing \partial$ has a codimension one component.

Then (up to scaling by a constant) the semi-simple part of $\partial$ has non-negative integer eigenvalues.
\end{lemma}
\begin{proof}
Let us write $\partial = \partial_S+\partial_N$ as the sum of its semi-simple and nilpotent parts.  Since $\partial$ is log canonical, 
Proposition \ref{prop_lc_iff_nonnilp} implies that $\partial_S \neq 0$. 
Lemma \ref{l_jordan} implies that if $\partial$ has a codimension one component of its zero locus, then the same is true of $\partial_S$.
So, without loss of generality we may freely replace $\partial$ by its semi-simple part and so may assume that $\partial$ is semi-simple, and after a suitable change of coordinates we may assume that $\partial = \sum \lambda_i x_i \partial_{x_i}$ where the $x_i$'s form a basis of the Zariski tangent space at $x$ and $\lambda_i\in \mathbb C$.

Let $D$ be an irreducible codimension one component of $\sing \partial$.  Up to relabelling the $x_i$ we may assume that $\ord_D x_1  = k_1 \ge 1$. Since $D\subset \sing \partial$, we have that $\lambda_1\neq 0$. Thus, 
up to rescaling by a constant, we may assume that $\lambda_1 = k_1$.

We will show that if $k_i\coloneqq \ord_D x_i \ge 0$ then $\lambda_i = k_i$.  Indeed, in $\mathcal O_{X, D}$ we have $x_i^{k_1} = u_i x_1^{k_i}$ where $u_i$ is a unit 
in $\mathcal O_{X, D}$ for any $i\ge 2$.  Applying $\partial$ to both sides we see that 
\[k_1\lambda_i x_i^{k_1} = \partial(x_i^{k_1}) = \partial(u_i x_1^{k_i}) = k_1k_i u_ix_1^{k_i}+x_1^{k_i}\partial(u_i) = k_1k_ix_i^{k_1}+x_1^{k_i}\partial(u_i).\]
Since $D$ is a codimension one zero locus of $\partial$ it follows that $\ord_D\partial(u) \ge 1$ and so subtracting the far right term from the far left term we see that 
\[\ord_D((k_1\lambda_i-k_1k_i)x_i^{k_1})) = \ord_D(x_1^{k_i}\partial(u)) \ge k_1k_i+1 > k_1k_i =\ord_D x_i^{k_1}\]
which implies that $k_1(\lambda_i-k_i) = 0$ as required.
\end{proof}

We have a similar characterisation of non-singular vector fields from the perspective of the MMP.

\begin{theorem}\label{t_terminal}
Let $X$ be a normal variety and let $\mathcal F$ be a rank one foliation on $X$ such that 
$K_{\mathcal F}$ is Cartier. Let $P \in X$ be a closed point.  

Then the following are equivalent:

\begin{enumerate}
\item   $\mathcal F$ is terminal at $P$. 

\item $P$ is not contained in $\sing \mathcal F$.

\item There is an analytic open neighbourhood $U$ of $P$, 
and a holomorphic submersion $f\colon U \rightarrow B$
such that $\mathcal F\vert_U$ is induced by $f$.

\item $P$ is not invariant by $\mathcal F$.
\end{enumerate}
\end{theorem}

\begin{proof}
See \cite[Lemma 2.9]{CS20}.
\end{proof}

We now consider a generalisation of Theorem \ref{t_terminal} to pairs:

\begin{lemma}
\label{lem_ev_trans}
Let $X$ be a normal variety. 
Let $\mathcal F$ be a rank one foliation on $X$ such that $K_{\mathcal F}$ is Cartier, let $D$ be a non-zero reduced Cartier divisor and let $x \in X$ be a point.

Then the following are equivalent:
\begin{enumerate}
\item  There exist suitable holomorphic coordinates near $x$ such that 
$X \cong \mathbb D \times Y$,
where $\mathbb D$ is the unit disc, 
 $\mathcal F$ is induced by $\mathbb D \times Y \to Y$ and $D = \{0\} \times Y$. 
\item There exists a local generator $\partial$ of $\mathcal F$ and a local equation $f = 0$ of $D$ such that $\partial(f) \in \mathcal O_{X, x}$ is a unit.
\item $(\mathcal F, D)$ is log canonical in a neighbourhood of $x \in X$.
\end{enumerate}
\end{lemma}
\begin{proof}
This is essentially \cite[Fact I.1.7]{McQuillan08}, but we briefly explain the main ideas now.  (1) easily implies (3).  (3) implies (2) by \cite[Remark 2.4]{Panazzolo}. Indeed, suppose for the sake of contradiction that $\partial(f) \in \mathfrak m$ and so if $\delta = f\partial$ we have $\delta(f) \in \mathfrak m^2$.  On one hand, since $(\mathcal F, D)$ is log canonical, the vector field $\delta$ has non-nilpotent linear part.  On the other hand, for any $h \in \mathcal O_{X, x}$, $\delta(h) \in (f)$ and $\delta(hf) = f\delta(h) + h\delta(f) \in \mathfrak m^2$, which implies that for any $h \in \mathfrak m$, \[
\delta^2(h) = \delta(h'f) = f\delta(h')+h'\delta(f) \in \mathfrak m^2\] and so the linear part of $\delta$ is nilpotent, a contradiction.  Finally, to see that (2) implies (1) we note that 
if $\partial(f)$ is a unit, then $\partial$ is non-singular and, after possibly rescaling $\partial$, we may assume that $\partial f=1$. 
 Hence 
in suitable coordinates we have $X \cong \mathbb D \times Y$ and $\partial= \partial_t$ where $t$ is the parameter on $\mathbb D$. 
Since $\partial_t f=1$, 
it follows that $f = t+\phi(y)$ where $\phi$ does not depend on $t$. 
Setting $t':=f=t+\phi(y)$ gives a further change of coordinates, after which
$D=\{t'=0\}$ and $\partial=\partial_{t'}$, proving (1).
\end{proof}

Let $X$ be a normal variety. 
Let $\mathcal F$ be a rank one foliation on $X$ such that $K_{\mathcal F}$ is $\mathbb Q$-Cartier, let $D$ be a reduced $\mathbb Q$-Cartier divisor.
If, up to a quasi-\'etale cover, any one of the equivalent conditions of Lemma \ref{lem_ev_trans} holds, we say that $D$ is {\bf transverse} to $\mathcal F$ at $x$.
If this holds for all $x \in X$ we say that $D$ is {\bf everywhere transverse} to $\mathcal F$.

\begin{remark}
\label{rem_1}
Let $\mathcal F$ be a rank one foliation such that $K_{\mathcal F}$ is $\mathbb Q$-Cartier and let $D$ be an effective Weil divisor such that $D$ is $\mathbb Q$-Cartier.
 There exists a quasi-\'etale cover $\sigma \colon Y \to X$ such that $\sigma^*K_{\mathcal F}$ and $\sigma^*D$
are both Cartier.  Since log canonicity is preserved by quasi-\'etale covers, cf. \cite[Proposition 5.20]{KM98} and \cite[Lemma 2.20]{SS23},
Lemma \ref{lem_ev_trans} shows that if $(\mathcal F, D)$ is log canonical, then $D$ is transverse to $\mathcal F$ at $x$.
In particular, if $(\mathcal F, D)$ is log canonical, then $D$ is reduced and each connected component of $D$ is irreducible.
\end{remark}

\begin{remark}
\label{rem_2}
We remark that Proposition \ref{prop_lc_iff_nonnilp}, 
Lemma \ref{lem_codim_one_zero_ss}, 
Theorem \ref{t_terminal}, 
Lemma \ref{lem_ev_trans} and Remark \ref{rem_1}
all hold in the setting where $X$ is the spectrum of a complete local ring (or the formal completion of a variety $X$ at a  point $x \in X$) by replacing ``holomorphic'' by ``formal'' and by replacing the unit disc $\mathbb D$
by its formal completion at $0 \in \mathbb D$.
\end{remark}

Of course, if $K_{\mathcal F}$ and $D$ are Cartier then $K_{\mathcal F}+D$ is Cartier, however frequently in MMP type arguments it is only natural to ask that $K_{\mathcal F}+D$ is Cartier
(or $\mathbb Q$-Cartier). We explain 
a generalisation of Lemma \ref{lem_ev_trans} to this situation:

\begin{proposition}
\label{prop_codim_1_zero}
Let $x \in X$ be a germ of a normal variety, let $\mathcal F$ be a rank one foliation and let $D$ be a reduced Weil divisor such that $K_{\mathcal F}+D$ is Cartier and $(\mathcal F, D)$ is log canonical.   

Then the following hold:
\begin{enumerate}
\item There exists a small modification $b\colon X' \to X$ such that $D'\coloneqq b_*^{-1}D$ is $\mathbb Q$-Cartier and $-D'$ is $b$-ample.  In particular, $(\mathcal F' \coloneqq b^{-1}\mathcal F, D')$ is log canonical, i.e., $D'$ is everywhere transverse to $\mathcal F'$.

\item There exists a generator $V \in T_{\mathcal F}(-\log D)$ (i.e. a section of $T_{\mathcal F}$ which vanishes along $D$) such that the associated derivation $\partial:=\partial_V$ is a radial vector field, i.e.
in appropriate (holomorphic) coordinates $\partial = \sum n_ix_i\partial_{x_i}$ where $n_i \in \mathbb Z_{\ge 0}$.
\end{enumerate}
\end{proposition}
\begin{proof}
Let us first choose a generator $V$ of $T_{\mathcal F}(-\log D)$ and write $\partial:=\partial_V = \partial_S+\partial_N$ as the sum of its semi-simple and nilpotent parts.  
Let $x_1, \dots, x_N$ be a basis of the Zariski tangent space of $x \in X$ so that we have an embedding $X \subset \mathbb C^N$ and a lift of $\partial$ to a vector 
field on $\mathbb C^N$.
By Lemma \ref{lem_codim_one_zero_ss} we know that (up to rescaling by a constant), in appropriate (formal) coordinates $\partial_S = \sum_{i = 1}^q n_i x_i\partial_{x_i}$ where $n_i \in \mathbb Z_{>0}$.
Note that $\{x_1 = \dots = x_q = 0\} \subset \sing \partial $.

By \cite[Lemma 6.2]{KM98}, the first claim of Item (1) is equivalent to  the $\mathcal O_X$-algebra $\bigoplus_{m \ge 0} \mathcal O_X(-mD)$ being finitely generated.  We will thus verify that 
this algebra is finitely generated.
Let $\hat X := \Spec \widehat{\mathcal O}_{X, x}$ where $\widehat{\mathcal O}_{X, x}$ 
is the completion of $\mathcal O_{X, x}$ at its maximal ideal.  Let $b\colon X' \to \hat X$ be the blow-up in $\{x_1 = \dots =x_q = 0\}$ with weights $n_1, \dots, n_q$ respectively and let $D'\coloneqq b_*^{-1}D$. 
Denote by $b'\colon \widetilde{\mathbb C^N} \to \hat{\mathbb C}^N$ the blow-up of $\hat{\mathbb C}^N$ in $\{x_1 = \dots = x_q = 0\}$ with weights $n_1, \dots, n_q$, where $\hat{\mathbb C}^N  := \Spec \mathbb C[[x_1, \dots, x_N]]$.
Let $E = \exc b'$ and let $F = E|_{X'}$.  A direct calculation as in Example \ref{ex_blow_up} shows that
$b'^*(\sum_{i = 1}^q n_i x_i\partial_{x_i})$ vanishes along $E$. Hence, 
$\delta := b^*\partial_S$ vanishes along $F$.  We claim that $\delta$ defines a foliation $\mathcal G$ on 
$X'$ 
such that  $(\mathcal G, F)$ is log canonical.  Indeed, 
by Proposition \ref{prop_lc_iff_nonnilp},
the vector field $\partial_S$ defines a log canonical foliated pair $(\mathcal F_S, D_S)$ on $\hat X$ and $K_{\mathcal G}+F = b^*(K_{\mathcal F_S}+D_S)$ by construction, hence $(\mathcal G, F)$ is log canonical.

  By assumption $D'$ is contained in the support of $ F$ and since $F$ is $\mathbb Q$-Cartier, Remark \ref{rem_1} implies that $F$ is irreducible.
  We deduce from this that $b\colon X' \to \hat X$ does not extract any divisors, i.e.  $b$ is a small modification, and that $D' = F$.
By \cite[Lemma 6.2]{KM98}, it follows that 
 $\bigoplus_{m \ge 0} \mathcal O_{\hat{X}}(-mD)$ is finitely generated.  Thus, \cite[Proposition 6.6]{KM98}
implies that $\bigoplus_{m \ge0} \mathcal O_X(-mD)$ is finitely generated, as required.
  The claim that $(\mathcal F', D')$ is log canonical follows because it is a small, and hence crepant, modification of a log canonical pair.  This proves  (1).

We now prove  (2). 
The fact that $\partial$ can be put into the claimed form (with a holomorphic change of coordinates) follows by observing that in an analytic neighbourhood $U$ of $D'$, 
$\mathcal F'$ admits a holomorphic first integral $g\colon U \to D'$.  Note that in fact since every $\mathcal F$-invariant subvariety is also $\mathcal F_S$-invariant, we 
have that $\mathcal F_S$ is tangent to the fibres of $g$, i.e., we have that $V_{\partial_S}$ is a (formal) section of $T_{\mathcal F}$. The existence of $g$ allows us to find holomorphic eigenfunctions of $\partial_S$, which in turn imply the existence of holomorphic coordinates as claimed.
\end{proof}

\begin{example}
We give two examples of log canonical vector fields on singular varieties with codimension one components of the zero set:
\begin{enumerate}
\item Fix $k\ge 2$. Let $X = \{xy+z^k  = 0\} \subset \mathbb C^3$ and let $\partial = kx\partial_x+z\partial_z$. 

\item Let $X = \{xy+zw = 0\} \subset \mathbb C^4$ and let $\partial = x\partial_x+z\partial_z$.  Note that in this case $\sing \partial$ is not a $\mathbb Q$-Cartier divisor.

\end{enumerate}
\end{example}

\subsection{Singularities of co-rank one foliations}
For co-rank one foliations, a suitable distinguished class of singularities has been identified: simple singularities.  Intuitively, foliations with simple singularities enjoy many of the nice
properties one would expect from varieties together with a simple normal crossing divisor.

One way to define simple singularities for co-rank one foliations is in terms of formal normal forms, cf.
\cite[Appendix]{cano04}:

\begin{definition}
\label{defn_simple}
Let $\mathcal F$ be a co-rank one foliation on a smooth variety $X$ of dimension $n$. 
We say that $p \in X$ is a {\bf simple singularity} for $\mathcal F$
provided that, in formal coordinates $x_1,\dots,x_n$ around $p,$ $N^*_{\mathcal F}$ is generated by a $1$-form which is 
in one of the following two forms, for some $1 \leq r \leq n$:

\begin{enumerate}
\item There are $\lambda_1,\dots,\lambda_r \in \mathbb C^*$, which satisfy the non-resonant condition and such that
$$\omega = x_1\cdots x_r\cdot\sum_{i = 1}^r \lambda_i \frac{dx_i}{x_i}.$$

\item There is an integer $k \leq r$ such that
$$\omega = x_1\cdots x_r\cdot \left(\sum_{i = 1}^kp_i\frac{dx_i}{x_i} + 
\psi(x_1^{p_1}\cdots x_k^{p_k})\sum_{i = 2}^r \lambda_i\frac{dx_i}{x_i}\right)$$
where $p_1,\dots,p_k$ are positive integers without a common factor, $\psi(s)$
is a formal power series which is not a unit, and the numbers $\lambda_2,\dots,\lambda_r \in \mathbb C^*$ satisfy the non-resonant condition.  
\end{enumerate}
\end{definition}

We say that the numbers  $\lambda_1,\dots,\lambda_l\in \mathbb C^*$ satisfy the {\bf non-resonant condition} if 
for any non-negative integers $a_1,\dots,a_l$ such that 
 $\sum a_i\lambda_i = 0$ we have that $a_i = 0$ for all $i=1,\dots,l$.

\medskip 

We have the following: 
\begin{proposition}
Let $X$ be a smooth variety and let $\mathcal F$ be a foliation with simple singularities.  

Then $\mathcal F$ has canonical singularities.
\end{proposition}
\begin{proof}
See \cite[Lemma 2.9]{CS21}.
\end{proof}
One important feature of simple singularities is that we have a very clear understanding of their separatrices. By a {\bf separatrix} of a co-rank one foliation $\mathcal F$ at a point $x \in X$ we mean a (possibly formal) divisor $x \in D$
such that $D$ is invariant by $\mathcal F$.  In the case that $D = \{f = 0\}$ and $\mathcal F$ is defined by a $1$-form $\omega$, the invariance condition can be phrased as the fact that $f$ divides $df\wedge \omega$.

In general, separatrices need not exist for co-rank one foliations.  One of the key ingredients in the study of the MMP for co-rank one foliations on threefolds is the fact that foliations with canonical singularities (on threefolds) always admit separatrices.
Indeed, the existence of separatrices for canonical foliations on threefolds allows to show that any curve $C \subset X$ tangent to $\mathcal F$ can be ``prolonged'' to a (possibly formal) invariant divisor $S$ containing $C$.  This is done by ``gluing'' separatrices of $\mathcal F$ for each $p \in C$ in a coherent way. We remark that this gluing is a delicate point:
a priori a formal separatrix is only defined at the formal completion of $X$ at a point $x$, and if a curve $C$ through $x$ is tangent to the foliation it is not clear if it is possible to extend the formal separatrix from the formal completion at $x$ to the formal completion at $C$.  Frequently, however, this is possible owing to an explicit analysis of the properties of simple singularities (see \cite[\S IV]{CC92} or \cite[\S 3.2]{CS21} for some results along these lines). 
The analysis of the pair $(X, S)$ is what allows us to construct the flip.

The existence of separatrices follows from two main ingredients.  The first is the existence of separatrices for {\bf non-dicritical singularities}
due to \cite{CC92} (we refer to \cite{CC92} for the definition of non-dicriticality for co-rank one foliation singularities).  This is a deep result and relies heavily on the reduction of singularities to simple singularities, and thus only applies in dimension at most $3$, although a version of this result has been achieved in arbitrary dimensions in \cite{CM92}.
The second main ingredient is that co-rank one canonical threefold singularities are non-dicritical.  This is proven in \cite[Theorem 1.6]{CS21} and uses the MMP in a central way.  It would be nice to prove this result in all dimensions.

We should also remark that the ``correct'' definition of non-dicritical
singularity for foliations of arbitrary rank, especially on singular varieties, is still obscure, but we refer to \cite[\S 4]{CC25} for some interesting results in this direction.

\begin{question}
    What is the correct definition of non-dicritical foliation singularity for foliations of arbitrary rank?
\end{question}

\begin{question}
\label{q_non_dic}
Let $X$ be a smooth (or normal) variety and let $\mathcal F$ be a foliation on $X$.  
Suppose that $\mathcal F$ has canonical singularities.  Does it have non-dicritical singularities?
\end{question}

We refer to \cite[Conjecture 4.2]{CS23a} and \cite[Theorem 2.1.9]{chen2023mmp} for closely related questions in the setting of algebraically integrable foliations; see also \cite[Corollary B]{bongiorno2026}. 

We remark that for rank one foliations there is a clear definition of non-dicritical singularity: on any birational modification, every exceptional divisor is invariant by the transformed foliation.  With this definition of non-dicritical, the answer to Question \ref{q_non_dic} is yes for rank one foliations in any dimension:

\begin{theorem}
\label{t_nd}
Let $X$ be a normal variety and let $\mathcal F$ be a rank one foliation on  $X$ with canonical singularities. 

Then $\mathcal F$ is non-dicritical. 
\end{theorem}
\begin{proof}
The question is local about any point $x \in X$, so we may freely replace $X$ by a neighbourhood of $x$ and we may assume that there exists $m>0$ such that $mK_{\mathcal F} \sim 0$. We may then take the index one cover associated to $K_{\mathcal F}$, $\sigma\colon X' \to X$ and note that $\sigma$ is quasi-\'etale.  Observe that $\sigma^{-1}\mathcal F$ is non-dicritical if and only if $\mathcal F$ is non-dicritical and that $\sigma^{-1}\mathcal F$ has canonical singularities. We may then apply  \cite[Corollary III.i.4]{MP13} to $\sigma^{-1}\mathcal F$ to conclude.
\end{proof}

\subsection{Singularities of algebraically integrable foliations}

Given a normal variety $X$ with an algebraically integrable foliation $\mathcal F$ such that $\mathcal F$ is induced by a rational map $f\colon X \dashrightarrow Z$, the problem 
of reducing the singularities of $\mathcal F$ is equivalent to that of reducing the singularities of the rational map $f$.  The natural target class of singularities is toroidal singularities, and due to 
\cite[Theorem 2.1]{AK00} we have the following reduction statement:

\begin{theorem}\label{t_AK}
Let $(X,B)$ be a  pair and let $f\colon X\dashrightarrow Z$ be a rational map between normal varieties.

Then there exist toroidal pairs   $(X',\Sigma_{X'})$ and $(Z',\Sigma_{Z'})$ and a diagram 

\[
\begin{tikzcd}
  X' \arrow[r, "\beta"] \arrow[d, "f'"'] & X \arrow[d, "f"] \\
  Z' \arrow[r, "\alpha"] &  Z
 \end{tikzcd}
\]

such that 
\begin{enumerate}
\item $\alpha$ and $\beta$ are birational projective morphisms,  
\item $f'\colon (X,\Sigma_X)\to (Z,\Sigma_Z)$ is a toroidal morphism;
\item $(Z',\Sigma_{Z'})$ is log smooth; and
\item the support of $\beta^{-1}_*B+\exc \beta$ is contained in $\Sigma_{X'}$.
%
\end{enumerate}
\end{theorem}

The relevance of toroidal singularities to the study of algebraically integrable foliations is explained by the following.

\begin{proposition}
Let $X$ be a normal variety and let $\mathcal F$ be a foliation.  Suppose that $\mathcal F$ is induced by a toroidal morphism $X\to Z$.

Then $\mathcal F$ admits log canonical singularities.
\end{proposition}
\begin{proof}
    See \cite[Lemma 3.1]{ACSS}.
\end{proof}
As a Corollary of the previous results, we obtain that if $\mathcal F$ is an algebraically integrable foliation on a normal variety $X$ then there exists a birational morphism $\beta\colon X'\to X$ such that $\beta^{-1}\mathcal F$ admits canonical singularities (see \cite[\S 3]{ACSS} 
and \cite[Proposition 5.6 and Theorem 5.12]{CC25}
for more details).

\section{Adjunction}
\label{s_adjunction}

Given a foliation $\mathcal F$ of rank $r$ on a smooth variety $X$ and a smooth codimension one
subvariety $D\subset X$ we can define a {\bf restricted foliation} $\mathcal F_D$ on $D$ (see Theorem \ref{t_adjunction} for a more general version).
Roughly speaking, the leaves of $\mathcal F_D$ are the leaves of $\mathcal F$ intersected
with $D$.  Thus, $\mathcal F_D$ is a foliation of rank $r-\epsilon(D)$.

In analogy with the {\bf adjunction formula} for singular
varieties (e.g. see \cite[\S 16]{Kollaretal}), we introduce a correction term $\mathrm{Diff}_D(\mathcal F) \geq 0$, 
called the {\bf different}.  With this correction term, we have a linear equivalence
\[(K_{\mathcal F}+\epsilon(D)D)\vert_D \sim K_{\mathcal F_D}+\mathrm{Diff}_D(\mathcal F).\]
We will show (cf. Proposition \ref{p_diff})  that if $P$ is a prime divisor contained in $\supp D \cap \sing \mathcal F$, then $\mult_P \Diff_D(\mathcal F)>0$.

\begin{theorem}[Adjunction]\label{t_adjunction} Let $X$ be a normal variety and let $\mathcal F$ be a foliation  of rank $r$ on $X$ such that $K_{\mathcal F}$ is $\mathbb Q$-Cartier.  Let $D\subset X$ be a  codimension one
subvariety such that $\epsilon(D)D$ is $\mathbb Q$-Cartier and let $\nu\colon W\to D$ be its normalisation. 

Then there exists a foliation $\mathcal F_W$ on $W$ of rank $r-\epsilon(D)$ and a $\mathbb Q$-divisor $\mathrm{Diff}_D(\mathcal F)\ge 0$ on $W$
such that 
\[(K_{\mathcal F}+\epsilon(D)D)\vert_W \sim K_{\mathcal F_W}+\mathrm{Diff}_D(\mathcal F).\]
\end{theorem}

\begin{proof}[Sketch of the proof]
For simplicity we assume that $W$ is normal (i.e. $W=D$) and that $K_{\mathcal F}$ and $D$ are Cartier. We refer to \cite{CS25} for the more general proof.  
We distinguish two cases:

\medskip 

We first assume that $D$ is $\mathcal F$-invariant.  By Lemma \ref{l_invariant} the map $T_{\mathcal F}|_D\to {T_X}|_D$ factors through $T_D$. Let $\eta\colon T_{\mathcal F}|_D \to T_D$ be the induced map and let $T_{\mathcal F_D}$ be the image of $\eta$. Around a general point of $D$, we have a commutative diagram
\[
\begin{tikzcd}
T_{\mathcal{F}}|_D \arrow[r] \arrow[d] & T_X|_D \\
T_{\mathcal{F}_D} \arrow[r] & T_D \arrow[u]
\end{tikzcd}
\]
Then $T_{\mathcal F_D}\subset T_D$ defines a foliation on $D$ of rank $r$. This induces a commutative diagram on $D$: 
\[
\begin{tikzcd}
\Omega^r_X|_D \arrow[r] \arrow[d] & \mathcal{O}_D(K_{\mathcal{F}}) \\
\Omega^r_D \arrow[r] & \mathcal{O}_D(K_{\mathcal{F}_D}) \arrow[u]
\end{tikzcd}
\]
In particular there is a non-trivial map $\mathcal{O}_D(K_{\mathcal{F}_D})\to  \mathcal{O}_D(K_{\mathcal{F}})$. Thus, there exists an effective divisor $\Diff_{D}(\mathcal F)\ge 0$ on $D$ such that
$$K_{\mathcal F}\vert_D \sim K_{\mathcal F_D}+\mathrm{Diff}_D(\mathcal F)$$
and our claim follows.

\medskip 

Now suppose that $D$ is not $\mathcal{F}$-invariant. 
By Lemma \ref{l_invariant} the composition
\[
\mathcal{O}_{D}(-D) \;\longrightarrow\; \Omega^1_X|_{D} \;\longrightarrow\; T^*_{\mathcal{F}}|_{D}
\]
is non-zero. Let $\mathcal E$ be its cokernel and let $T_{\mathcal F_D} := \mathcal E^*$. Then, $\rk T_{\mathcal F_D} =\rk \mathcal F-1$ and 
the induced inclusion
$T_{\mathcal F_D} \subset T_D$
satisfies the integrability condition.

Let $\phi\colon \Omega^{[r]}_X\to \mathcal O_X(K_{\mathcal F})$  be the Pfaffian map. 
We define a morphism
\[
\psi \colon \Omega_D^{r-1} \;\longrightarrow\;
\mathcal{O}_D(K_{\mathcal{F}}  + D)
\]
by
\[
\psi(\alpha) := \left.\frac{\phi(df \wedge \tilde\alpha)}{f}\right|_D
\]
for any local section $\alpha$ of $\Omega^r_D{}^{-1}$, where $f$ is a local equation of $D$, 
$\tilde\alpha$ is any local lift of $\alpha$ to $\Omega^r_X{}^{-1}$, and 
$\frac{\phi(df \wedge \tilde\alpha)}{f}$ is considered as a section of 
$\mathcal{O}_X(K_{\mathcal{F}} + D)$.

We claim that this morphism is independent of our choice of $f$. Indeed, if $f'$ is another local equation of $D$ then 
$f' = u f$ where $u$ is a unit. We compute that
\[
\frac{\phi(df' \wedge \tilde\alpha)}{f'}
= \frac{\phi(df \wedge \tilde\alpha)}{f} + \phi(\omega_0),
\]
where $\omega_0$ is a local section of $\Omega^r_X$. Observe that $\phi(\omega_0)$ is a section of 
$\mathcal{O}_X(K_{\mathcal{F}})$, and so it vanishes along $D$ when considered as a section of 
$\mathcal{O}_X(K_{\mathcal{F}}  + D)$. Thus,
\[
\left.\frac{\phi(df' \wedge \tilde\alpha)}{f'}\right|_D
= \left.\frac{\phi(df \wedge \tilde\alpha)}{f}\right|_D,
\]
as required. Likewise, it is easy to check that $\psi$ is independent of the choice of the lift $\tilde\alpha$.

By construction, $\psi$ factors through 
the Pfaffian map $\Omega^{r-1}_D\to \mathcal O_D(K_{\mathcal F_D})$. Thus, there exists a non-zero morphism $ \mathcal O_D(K_{\mathcal F_D}) \to \mathcal O_D(K_{\mathcal F}+D)$, which implies that there exists 
an effective divisor $\Diff_{W}(\mathcal F)\ge 0$ such that
$$(K_{\mathcal F}+D)\vert_D \sim K_{\mathcal F_D}+\mathrm{Diff}_D(\mathcal F)$$
and our claim follows again. 
\end{proof}

 The adjunction formula for varieties says a bit more: it gives a way to relate the singularities of the pair $(X, D)$ to the singularities of the pair $(D, \Diff_D(X))$.  In general, we do not have such control for a foliation when $D$ is $\mathcal F$-invariant, but see Theorem \ref{thm_adj_sing} below.  The following examples show that even if $\mathcal F$ is canonical, it is not true that $(\mathcal F_D, \mathrm{Diff}_D(\mathcal F))$ is log canonical:
\begin{example}
\begin{enumerate}
\item  Let $X=\mathbb C^2$, fix a positive integer $m\ge 2$ and let $\mathcal F$ be the rank one foliation defined by the derivation 
$$\partial := x\partial_x + y^m \partial_y.$$
Then $D:=\{x=0\}$ is $\mathcal F$-invariant and $T_{\mathcal F_D}=T_D$. In this case, $\Diff_D(\mathcal F)=m\cdot p$ where $p=(0,0)$. Thus, $(\mathcal F_D,\Diff_D(\mathcal F))$ is not log canonical. 

\item Let $X=\mathbb C^3$, fix a positive integer $m\ge 2$ and let $\mathcal F$ be the rank one foliation defined by the derivation 
$$\partial := x\partial_x + y^m \partial_y+z^m\partial_z.$$
Then $D:=\{x=0\}$ is $\mathcal F$-invariant and ${\mathcal F_D}$ is the foliation defined by the vector field
$$ y^m \partial_y+z^m\partial_z.$$
In this case, $\Diff_D(\mathcal F)=0$ but $\mathcal F_D$ is not log canonical. 

\end{enumerate}

\end{example}

\begin{theorem}
\label{thm_adj_sing}
Let $X$ be a normal variety, 
let $\mathcal F$ be a foliation of rank $r$  on $X$ and let $D$ be a prime divisor which is not 
$\mathcal F$-invariant.
Suppose that $K_{\mathcal F}$ and $D$ are $\mathbb Q$-Cartier.
Suppose that $(\mathcal F, D)$
is log canonical.  Let $\nu \colon W \rightarrow D$ be the normalisation.

Then $(\mathcal F_W, \Diff_D(\mathcal F))$
is log canonical.
\end{theorem}
\begin{proof}
    See \cite[Theorem 3.16]{CS23b}.
\end{proof}

\begin{proposition}\label{p_diff}
Let $X$ be a normal variety, let $\mathcal{F}$ be a foliation of rank $r$ on $X$, let $D$ be a prime divisor on $X$. Suppose that $K_{\mathcal{F}}$ and $\epsilon(D)D$ are $\mathbb{Q}$-Cartier. Let $n \colon S \to D$ be the normalisation, let $\mathcal{F}_S$ be the restricted foliation on $S$ and let $P$ be a prime divisor on $S$.  

If $n(P)$ is contained in $\operatorname{Sing}\mathcal{F}$,  then $\mult_P \operatorname{Diff}_S(\mathcal{F}) > 0$.
 \end{proposition}

\begin{proof}[Sketch of the Proof] 
For simplicity, we assume that $D$ is normal and so $S=D$. 
We refer to \cite{CS25} for the more general proof. 

Let us first suppose that $K_{\mathcal{F}}$ and $\epsilon(D)D$ are Cartier in a neighbourhood of $P$. By assumption, the morphism 
\[
\phi \colon \Omega^r_X \to \mathcal{O}_X(K_{\mathcal{F}})
\]
takes values in $I \mathcal{O}_X(K_{\mathcal{F}})$ where $I$ is an ideal sheaf, whose co-support contains $P$.

If $\epsilon(D) = 0$ then  the restriction of $\phi$ to $D$ factors through
\[
\Omega^r_D \to \mathcal{O}_D(K_{\mathcal{F}} - P).
\]

If $\epsilon(D) = 1$ then, as in the proof of Theorem \ref{t_adjunction}, we may define 
\[
\psi \colon \Omega^{r-1}_X \to \mathcal{O}_X(K_{\mathcal{F}} + D)
\quad \text{by} \quad
\psi:= \frac{\phi(df \wedge \cdot)}{f}
\]
for some choice of a local parameter $f$ for $D$. Then $\psi$ still factors through $I \mathcal{O}_X(K_{\mathcal{F}} + D)$ and so  the restriction to $D$ factors through
\[
\Omega^{r-1}_D \to \mathcal{O}_D\bigl(K_{\mathcal{F}} + D - P\bigr).
\]

In both cases, we have that $\mult_P \operatorname{Diff}_D(\mathcal{F}) \geq 1$, as required.

\medskip

We now consider the general case. Let $q \colon V \to U$ be a quasi-\'etale cyclic cover, where $U$ is a neighbourhood of the generic point of $P$ in $X$, and such that $q^*K_{\mathcal{F}}$ and $q^*(\epsilon(D)D)$ are Cartier divisors. Let $\mathcal{F}' = \mathcal{F}_V$. 
By \cite[Proposition 5.13]{Druel21}, we have that 
$P$ is contained in $\operatorname{Sing}\mathcal{F}$ if and only if $P' := q^{-1}(P)$ is contained in $\operatorname{Sing}\mathcal{F}'$. Let $S'$ denote the normalisation of $q^{-1}(D)$ and let $\mathcal{F}_{S'}$ and $\operatorname{Diff}_{S'}(\mathcal{F}')$ be respectively the restricted foliation and the different associated to $\mathcal{F}'$ on $S'$. We have a finite morphism $p \colon S' \to S$ and we have that
\[
K_{\mathcal{F}_{S'}} + \operatorname{Diff}_{S'}(\mathcal{F}') = p^*(K_{\mathcal{F}_S} + \operatorname{Diff}_S(\mathcal{F})).
\]

Let $e$ be the ramification index of $p$ along $P'$. By Riemann-Hurwitz (cf. Proposition \ref{p_RH})
we get an equality
\[
\mult_P \operatorname{Diff}_{S'}(\mathcal{F}') = e\, \mult_P \operatorname{Diff}_S(\mathcal{F}) - \epsilon(\mathcal{F}_S, P)(e-1).
\]

Since $\mult_P \operatorname{Diff}_{S'}(\mathcal{F}') \geq 1$ this implies that $\mult_P \operatorname{Diff}_S(\mathcal{F}) > 0$.
\end{proof}

\section{MMP for surface foliations}
\label{s_mmp2}
In ~\cite{McQuillan08}, McQuillan developed  a minimal model program for foliations on surfaces. In this section we review this result, and explain how it follows as a simple consequence of our results on adjunction and the classical MMP for surfaces.

\subsection{MMP with less}

\begin{proposition}\label{p_discr}
Let $X$ be a normal $\mathbb Q$-factorial surface and let $\mathcal F$ be a rank one foliation on $X$. Assume that $C_1,\dots,C_n$ are $\mathcal F$-invariant curves on $X$ and let $\eta\colon Y\to X$ be a birational morphism. For any $\eta$-exceptional divisor $E$, denote by $b_E:=a(E,X,\sum_{i=1}^n C_i)$  the log discrepancy of $(X,\sum_{i=1}^n C_i)$ along $E$.

Then, 
\begin{enumerate}
\item If $\mathcal F$ is non-dicritical, we have that 
$b_E\ge a(E,\mathcal F)$ for any  $\eta$-exceptional divisor $E$.

\item If $\mathcal F$ admits canonical singularities then $(X,\sum_{i=1}^n C_i)$ is log canonical. 
\end{enumerate}
\end{proposition}
\begin{proof}
After possibly replacing $\eta$ by a higher model, we may assume that $\eta$ is a log resolution of $(X, \sum C_i)$.  Write 
$$K_{\mathcal F_Y}=\eta^*K_{\mathcal F}+\sum_{E \subset \exc \eta} a(E,\mathcal F) E.$$
We also have 
$$K_{Y}+\sum_{i=1}^n C'_i +\sum_{E \subset \exc \eta} E= \eta^*(K_X+\sum_{i=1}^n C_i) + \sum_{E \subset \exc \eta} b_E E,$$
where $C'_i$ is the strict transform of $C_i$ on $Y$ for $i=1,\dots,n$.
Fix any prime exceptional divisor $E_0$. 
Since $\mathcal F$ is non-dicritical, each $\eta$-exceptional divisor $E$ of $\eta$ is $\mathcal F_Y$-invariant. Thus, if $E\neq E_0$ then $E\cap E_0$ is contained in $\sing \mathcal F_{Y}$ and, in particular, by Proposition \ref{p_diff}, we have that $E\cap E_0$ is contained in the support of $\Diff_{E_0}(\mathcal F_Y)$. 
Similarly $C'_i\cap E_0$ is contained in the support of $\Diff_{E_0}(\mathcal F_Y)$ for each $i=1,\dots,n$. 
Thus, 
restricting to $E_0$ we get
$$
\begin{aligned}
\sum_{E \subset \exc \eta} b_E E\cdot E_0 &= \deg (K_Y+\sum_{i = 1}^n C'_i + \sum_{E \subset \exc \eta} E)|_{E_0}  \\
& = \deg (K_{E_0} + \bigl(\sum_{i = 1}^n C'_i + \sum_{E \subset \exc \eta, ~ E\neq E_0} E\bigr)|_{E_0}) \\
& \le \deg (K_{E_0} + \Diff_{E_0}(\mathcal F_Y)) \\
&=  K_{\mathcal F_Y}\cdot E_0 \\
&=\sum_{E \subset \exc \eta} a(E, \mathcal F) E \cdot E_0.
\end{aligned}
$$
Thus, \cite[Lemma 3.41]{KM98} implies (1) and Theorem \ref{t_nd} implies (2).
\end{proof}

\begin{proposition}
\label{p_contraction}
    Let $X$ be a normal $\mathbb Q$-factorial surface and let $\mathcal F$ be a non-dicritical foliation of rank one on $X$. Assume that $C_1,\dots,C_n$ are $\mathcal F$-invariant curves on $X$
    and that $C_1$ is projective.

    Then, \[ (K_X+\sum_{i=1}^n C_i )\cdot C_1 \leq K_{\mathcal F}\cdot C_1.\]

    In particular, if $\mathcal F$ has canonical singularities, $K_{\mathcal F}\cdot C_1<0$ and $C_1^2<0$ then there exists a contraction $\phi\colon X \to Z$ of $C_1$ to a point.
\end{proposition}
\begin{proof}
    Let $\eta\colon Y \to X$ be 
    a log resolution of $(X, \sum C_i)$, let $\mathcal F_Y$ be the foliation induced on $Y$ and let $C_i'$ be the strict transform of $C_i$ on $Y$.
    Since $\mathcal F$ is non-dicritical each $\eta$-exceptional divisor is $\mathcal F_Y$-invariant, hence for any $\eta$-exceptional divisor $E$ we have $E \cap C'_1 \subset \sing \mathcal F_Y$.
    Similarly $C'_i\cap C'_1\subset \sing\mathcal F_Y$ for any $i=2,\dots,n$.
    Thus, Proposition \ref{p_diff} implies
\begin{align}
\label{ineq_1}
    (K_Y+\sum_{i = 1}^n C_i'+\sum_{E \subset \exc \eta} E)\cdot C'_1 &= \deg (K_Y+\sum_{i = 1}^n C_i'+\sum_{E \subset \exc \eta} E)|_{C'_1} \\
  \notag  &= \deg(K_{C'_1}+\sum_{i=2}^nC'_i+\sum_{E \subset \exc \eta} E|_{C'_1} )\\
\notag    &\leq \deg(K_{C'_1}+\Diff_{C'_1}\mathcal F_Y) \\
  \notag   &= K_{\mathcal F_Y}\cdot C'_1.
\end{align}
Let $b_E:=a(E,X,\sum_{i=1}^n C_i)$. 
    Proposition \ref{p_discr}
    implies that 
\begin{align*}
    \big ((K_Y+\sum_{i = 1}^n C'_i+\sum_{E \subset \exc \eta}E) - &\eta^*(K_X+\sum_{i = 1}^n C_i)\big )\cdot C'_1 \\
    &= (\sum_{E \subset \exc \eta} b_EE)\cdot C'_1 \\
    &\ge (\sum_{E \subset \exc \eta} a(E, \mathcal F)E)\cdot C'_1 \\ 
    &= (K_{\mathcal F_Y}- \eta^*K_{\mathcal F})\cdot C'_1.
    \end{align*}
Rearranging gives 
\begin{align}
\label{ineq_2}
((K_Y+\sum_{i = 1}^n C_i'+\sum_{E \subset \exc \eta} E) - K_{\mathcal F_Y})\cdot C'_1
\ge (\eta^*(K_X+\sum_{i = 1}^n C_i)) - \eta^*K_{\mathcal F}\cdot C'_1.
\end{align}
    Hence by (\ref{ineq_1}) and (\ref{ineq_2}), it follows
    \[(K_{X}+\sum_{i = 1}^n C_i)\cdot C_1 = \eta^*(K_X+\sum_{i = 1}^n C_i)\cdot C'_1 \leq \eta^*K_{\mathcal F}\cdot C'_1 =K_{\mathcal F}\cdot C_1,\] as required.
    The existence of the claimed contraction then follows by noting that $(X, C_1)$ is log canonical by Proposition \ref{p_discr} and $(K_X+C_1)\cdot C_1 <0$, hence the contraction exists by \cite[Theorem 3.7]{KM98}.
\end{proof}

\begin{remark}
Note that the $\mathbb Q$-factoriality hypothesis in Proposition \ref{p_discr} and Proposition \ref{p_contraction} is not essential.  On an arbitrary normal surface we can rephrase the above results in the language of Mumford's intersection theory on surfaces and therefore drop the $\mathbb Q$-factorial condition.
\end{remark}

Proposition \ref{p_contraction} now allows us to run the MMP.  Since the general ideas of how to run the MMP have been surveyed before, we will only sketch the argument in broad outline.

\begin{theorem}
\label{thm_mmp_surface}
    Let $X$ be a normal $\mathbb Q$-factorial projective surface and let $\mathcal F$ be a rank one foliation with canonical singularities.
    Then either
    \begin{itemize}
        \item $\mathcal F$ is birational to a fibration in rational curves; or
        \item there exists a birational contraction $\psi\colon X \to X'$ such that $\psi_*K_{\mathcal F}$ is nef.
    \end{itemize}
\end{theorem}
\begin{proof}[Sketch of the Proof]
    If $K_{\mathcal F}$ is not pseudo-effective, then \cite{SB92} implies that $\mathcal F$ is birational to a fibration in rational curves.  We may therefore assume that $K_{\mathcal F}$ is  pseudo-effective.

    If $K_{\mathcal F}$ is nef, then there is nothing to show.
    So let us assume that $K_{\mathcal F}$ is not nef.  We will explain an iterative process to produce our desired contraction.
    
 Since $K_{\mathcal F}$ is not nef, there is a curve $C \subset X$ such that $K_{\mathcal F}\cdot C<0$. Since $K_{\mathcal F}$ is pseudo-effective we have that $C^2<0$.
    We claim that $C$ is $\mathcal F$-invariant.  Suppose to the contrary that it is not invariant.  Then by adjunction we have
    \[0 \leq \deg (\Diff_C(\mathcal F)) = (K_{\mathcal F}+C)\cdot C = K_{\mathcal F}\cdot C +C^2 <0,\]
    a contradiction.

    Since $C$ is $\mathcal F$-invariant we may apply Proposition \ref{p_contraction} to produce a contraction $\phi\colon X \to Z$ of $C$ to a point.  
    It is easy to show that $\phi_*\mathcal F$ has canonical singularities.  So we may replace $X$ and $\mathcal F$ by $Z$ and $\phi_*\mathcal F$ and continue this process.  Since each iteration drops $\rho(X)$
    by $1$, this process eventually terminates and its output is our desired model.
\end{proof}

\begin{definition}[cf. \cite{brunella00}, Chapt 5]
Given a  rank one foliation $\mathcal{F}$ with canonical singularities on a smooth surface $X$,
we say that a curve $C$ in $X$ is $\mathcal{F}$-{\bf exceptional} if
$C$ is a smooth rational curve of self-intersection $-1$ and
if $\pi\colon X \to \overline X$ is the contraction of $C$ to a point then the foliation $\overline{\mathcal F}$ induced on $\overline X$ has a canonical singularity at $p = \pi(C)$. 
We say that $\mathcal F$ is {\bf relatively minimal} if 
$\mathcal F$ admits only canonical singularities and there are no $\mathcal F$-exceptional curves on $X$. 
\end{definition}

Note that our definition is slightly different from the one in \cite{brunella00}, Chapt 5], as we only require $\mathcal F$ to admit canonical singularities instead of reduced (cf. \cite[paragraph after Definition 8.2]{brunella00}).

\begin{remark} \label{r_suppN}
Set-up as in Theorem \ref{thm_mmp_surface}.
    Suppose that $K_{\mathcal F}$ is pseudo-effective and let $K_{\mathcal F} = P+N$ be its Zariski decomposition.  The morphism $\psi$ given by Theorem \ref{thm_mmp_surface}
    contracts precisely $\supp N$ to a set of points.

    In fact, more can be said about $\supp N$ in the case that $\mathcal F$ is relatively minimal.
    Indeed, if $\mathcal F$ is relatively minimal, then $\supp N$ is a disjoint union of Hirzebruch-Jung chains whose components are $\mathcal F$-invariant curves. In particular, this implies that if $X$ is smooth and $\psi\colon X \to X'$ is contraction given by Theorem \ref{thm_mmp_surface} then $X'$ has only cyclic quotient singularities. We remark that this can also be seen as a direct consequence of Theorem \ref{t_terminal} and the Negativity Lemma. Moreover, if $C$ is an $\mathcal F$-invariant curve which is not contained in the support of $N$ then $C$ meets any connected component of $\supp N$ in at most one point and $C\cdot N\le \frac k2$ where $k$ is the number of connected components of $\supp N$ intersecting $C$. 
\end{remark}

\subsection{Foliations of general type}
For surfaces, more can be said about the structure of the minimal model when $\mathcal F$ is of general type, i.e., $K_{\mathcal F}$ is big.  Theorem \ref{t_zeroloc} (which is proven in
\cite[Theorem 2]{McQuillan08}) gives such a description.

We begin with two basic observations:

\begin{proposition}\label{zeroinv}
Let $\mathcal{F}$ be a foliation of general type on a smooth surface $X$.
Let $K_{\mathcal{F}}=P+N$ be the Zariski decomposition of $K_{\mathcal F}$ and
let $C$ be a curve such that $P\cdot C=0$.

Then $C$ is $\mathcal{F}$-invariant.
\end{proposition}
\begin{proof}
Assume for the sake of contradiction that $C$ is not $\mathcal F$-invariant. Let $\varepsilon\colon X\rightarrow Y$ be the contraction of all the components of $N$
and let $\overline K=\epsilon_*K_{\mathcal F}$ and $\overline C=\epsilon_*C$. Then,
$$\overline K\cdot \overline {C}=\epsilon^*\overline K\cdot  C =P\cdot C =0.$$
Since $\overline K$ is big,
we have $\overline{C}^2<0$.
On the other hand, adjunction implies
$$\overline{C}^2=(\overline K+\overline{C})\cdot \overline{C}\geq 0$$
that is a contradiction.
Thus, $C$ is $\mathcal{F}$-invariant.
\end{proof}

We will also need the following consequence of the Camacho-Sad formula.

\begin{proposition}
\label{Cinv}
Let $\mathcal{F}$ be a foliation on a smooth surface $X$ and  
let $C$ be a  curve on $X$ whose components are $\mathcal F$-invariant and such that $C \cap \sing \mathcal F = \emptyset$.

Then $C^2=0$.
\end{proposition}

\begin{proof}
This is an easy consequence of the Camacho-Sad formula, see \cite[Theorem 3.2]{brunella00} (see also \cite[Theorem 3.9]{pereira22} for a more general version).
\end{proof}

\begin{theorem}
\label{t_zeroloc}
Let $\mathcal{F}$ be a relatively minimal foliation of general type on a smooth surface $X$.
Let $K_{\mathcal{F}}=P+N$ be the Zariski decomposition of $K_{\mathcal F}$ and
let $Z$ be the union of all the curves $C$ such that  $P\cdot C=0$.

Then every curve in the support of $Z$ is $\mathcal F$-invariant and $Z$ is the disjoint union of:
\begin{enumerate}
\item $\supp N$;
\item disjoint chains of rational curves none of which is contained in $\supp N$;
\item configurations $\Gamma$ of rational curves whose dual graph is a cycle and such that ${\sing}(\mathcal{F})\cap \Gamma$ coincides with the singular locus of
$\Gamma$. 
\end{enumerate}
\end{theorem}


\begin{proof}[Sketch of the Proof]
Let $C$ be a curve such that $P\cdot C=0$ and which is not contained in the support of $N$.
By Proposition \ref{zeroinv}, the curve $C$ is $\mathcal{F}$-invariant and by adjunction we have 
$$K_{\mathcal{F}}\cdot C=2g(C)-2+\deg\Diff_{C}(\mathcal F).$$
Let $k$ be the number of points in which $C$ intersects the support of $N$.
Then by adjunction, Proposition \ref{p_diff}  and Remark \ref{r_suppN}, we have
$$\frac{k}{2}\geq N\cdot C=K_{\mathcal{F}}\cdot C=2g(C)-2+\deg\Diff_{C}(\mathcal F)\geq 2g(C)-2+k. $$
The inequality $2g(C)-2+k\leq k/2$
implies either 
\begin{itemize}
\item $g(C) = 1$ and $k = 0$; or \item $g(C)=0$ and $k\leq 4$.
\end{itemize}

We first show that the case $g(C)= 1$ does not occur.  Indeed, suppose for the sake of contradiction that $g(C) = 1$. Since $K_{\mathcal F}\cdot C = 0$ we see that $\deg\Diff_{C}(\mathcal F) = 0$.  From 
the Camacho-Sad formula, Proposition \ref{Cinv}, we deduce that $C^2 = 0$ which contradicts the assumption that $P$ is big.
So we may assume that $g(C) = 0$ and $k \le 4$.

Camacho-Sad formula also implies that $k<\deg\Diff_{C}(\mathcal F)$ (cf. \cite[Theorem 3.4]{brunella00}).
In particular, 
$$k< \deg\Diff_{C}(\mathcal F)\leq 2+ \frac{k}{2}$$
implies $k\le 2$. 
Since $N\cdot C=K_{\mathcal F}\cdot C$ has to be an integer, it follows that if $k=0,1$ then 
$\deg\Diff_{C}(\mathcal F)=2$.
Finally, if $k=2$ then $\deg\Diff_{C}(\mathcal F)=3$ and, since $N\cdot C=1$, it follows that
the curve $C$ intersects $N$ in two components of coefficient $1/2$.
We notice that the two components have to be distinct because
each intersection point is a singular point for the foliation.
To conclude, we notice that the only possibilities for the connected components of $Z$
have to be those listed.
\end{proof}

\begin{definition}
A curve as in Case (3) of Theorem  \ref{t_zeroloc}
is called an \textit{elliptic Gorenstein leaf} (e.g.l. for short).
\end{definition}




  Moreover, we have the following characterisation due to McQuillan, \cite[Theorem IV.2.2 and Corollary IV.2.3]{McQuillan08}:

\begin{theorem}
\label{t_failure}
Let $\mathcal{F}$ be a relatively minimal foliation of general type on a smooth surface $X$.
\begin{enumerate}
    \item Suppose that $X$ contains an e.g.l. $\Gamma$.  Then $K_{\mathcal{F}}\vert_{\Gamma}$ is not a torsion divisor.

    \item $R(X, K_{\mathcal F})$ is finitely generated if and only if there are no e.g.l.s.
\end{enumerate}

\end{theorem}

We remark that e.g.l.s do indeed exist: examples can be found by taking covers of minimal resolutions of Bailey-Borel compactifications of bidisk quotients, cf. \cite[Example II.2.3]{McQuillan08}. Thus, finite generation does not hold in the context of foliations, in contrast to the case of varieties (cf. \cite{BCHM}).

\begin{question}
    Is there a similar characterisation in higher dimensions of foliations of general type for which $R(X, K_{\mathcal F})$ is not finitely generated?
\end{question}

As a first step towards this question it might be worthwhile answering the following.

\begin{question}
    Is there a structural result like Theorem \ref{t_zeroloc} for foliations on threefolds?
\end{question}

\section{MMP in higher dimensions}
\label{s_flips}

In this section we review some features of the MMP for foliations in higher dimensions, with a particular focus on the threefold case.  Since the general structure of the MMP has been given many survey treatments before (see for instance \cite{Cascini21}),
we will not review it here, but rather focus on some of the particular tools and techniques used in the MMP.

We will especially focus on a new feature of the MMP which starts in dimension three, namely, the flip.

We will also need the following result.

\begin{theorem}[Reeb stability theorem]\label{t_reeb}
Let $X$ be a normal variety and let $\mathcal F$ be a rank one foliation on $X$. Let $C \subset X$
be an $\mathcal F$-invariant curve and suppose that  $\mathcal F$
is terminal at every closed point $P\in C$.  Suppose moreover that $K_{\mathcal F}\cdot C<0$.  

Then 
$C$ moves in a family of $\mathcal F$-invariant curves covering $X$.
\end{theorem}
\begin{proof}
See \cite[Proposition 3.3]{CS20}.
%
\end{proof}

\subsection{Existence of flips for rank one foliations on threefolds}

We recall that the existence of flipping contractions in the classical setting is a consequence of the base point free theorem, which is no longer valid in the same generality in the setting
of foliations. Rather, the flipping contraction must be constructed via more {\it ad hoc} methods.  
We will place these considerations to the side for the moment and consider the following setting:

Let $X$ be a projective $\mathbb Q$-factorial klt threefold and let $\mathcal F$ be a rank one foliation on $X$ with simple singularities (we refer to \cite[Definition 2.24]{CS20} for the definition 
of simple singularities for rank one foliations).
Let $f\colon X \to Z$ be the contraction of a $K_{\mathcal F}$-negative extremal ray $R \subset \overline{NE}(X)$ and suppose that $\exc f$ is of codimension two.

In this setting we have the following facts:
\begin{enumerate}
\item By the Cone theorem for rank one foliations (cf. \cite[Theorem 1.2]{CS23b}) we can deduce that $\exc f$ is a union of $\mathcal F$-invariant rational curves.

\item Thanks to Reeb stability theorem (cf. Theorem \ref{t_reeb}) and adjunction we know that for each irreducible component $C$ of $\exc f$ there is exactly one closed point $P \in C$
where $\mathcal F$ is not terminal.

\item It is possible to show that each connected component of $\exc f$ is smooth and irreducible.
\end{enumerate}

Since the construction of the flip is (analytic) local on the base, we may freely replace $Z$ by an analytic neighbourhood of $f(C) \in Z$.

Our first observation is that, after possibly replacing $Z$ by a neighbourhood of $f(C)$, $\mathcal F$ is $1$-complemented, in other words:

\begin{proposition}
\label{prop_complement}
Set-up as above.  There exists a reduced divisor $T$ such that $(\mathcal F, T)$ is log canonical and $K_{\mathcal F}+T \sim_{\mathbb Q} f^*M$
where $M$ is a Cartier divisor on $Z$.
\end{proposition}
\begin{proof}
    This is \cite[Proposition 5.6]{CS20}.  Alternatively, this is also a consequence of the description of the formal neighbourhood of flipping curves given in \cite[Proposition/Summary II.g.3]{McQ05}.
\end{proof}

Let us define $\mathcal G := f_*\mathcal F$ and $D:= f_*T$.

\begin{proposition} 
Set-up as above. 
In an analytic neighbourhood of $C$, $\mathcal F$
admits a meromorphic first integral.
\end{proposition}
\begin{proof}
    By Proposition \ref{prop_complement}
    $K_{\mathcal G}+D$ is Cartier,
    and so $T_{\mathcal G}(-\log D)$ is generated
by a log canonical vector field $\partial$.  Since $\partial$ has a codimension one zero, namely, $D$, we may apply Proposition \ref{prop_codim_1_zero}
to deduce that, in a Euclidean neighbourhood of $f(C)$, $\mathcal G$ admits a meromorphic first integral.
\end{proof}

With this meromorphic first integral in hand, it is then comparatively easy to construct the flip.  For instance, one can use this meromorphic first integral to produce a $\mathbb Q$-divisor $G \ge 0$
on $X$ such that $G$ is $\mathcal F$-invariant and $C$ is a log canonical centre of $(X, G)$, cf. \cite[Proposition 5.11]{CS20}.  Using some results comparing adjunction on foliations to subadjunction on log canonical
pairs, cf. \cite[Theorem 4.9]{CS20}, we can deduce that $(K_X+G)\cdot C<0$ and that $(X, G)$ is log canonical.  We can therefore construct the $K_{\mathcal F}$-flip by constructing the $(K_{X}+G)$-flip.
We have therefore proven the following (cf. \cite[Theorem 1.1]{CS20}):

\begin{theorem}
Set-up as above.  Then, the $K_{\mathcal F}$-flip exists.
\end{theorem}

\subsection{Existence of flips for co-rank one foliations on threefolds}

Similarly to the rank one case, 
in the co-rank one setting the flip is constructed by finding an $\mathcal F$-invariant divisor $G \ge 0$ such that the flipping curve $C$
is a log canonical centre of $(X, G)$.  
This $G$ is found by taking the union of all the separatrices (formal or otherwise) which meet $C$, 
and then prolonging them to a (formal) neighbourhood of $C$.  
Some additional technical complication is caused by the fact that these divisors are only defined on a neighbourhood of $C$ (and not on all of $X$) and may also only be formal divisors.  These additional complications can be handled by either using Artin/Elkik approximation theorems, or by using recent developments in the MMP for formal schemes, see \cite{lyu2025relativeminimalmodelprogram}.

Of course, in our constructions, the divisor $G$ depends on the flipping ray chosen.  It would be interesting to know if $G$ could be chosen to be independent of the flipping ray.

\begin{question}
Let $X$ be a klt variety and let $\mathcal F$ be a foliation with log canonical singularities.  Let $\phi\colon X \dashrightarrow X'$ be a sequence of $K_{\mathcal F}$-flips.  
Does there exist a divisor $G \ge 0$ on $X$ such that $(X, G)$ is log canonical and  $\phi$ is a sequence of $(K_X+G)$-flips?

\end{question}

Note that the answer to this question remains unclear, even under the assumption that $\mathcal F$ is algebraically integrable (see \cite[Proposition 3.7]{ACSS} for a special case).

\bibliography{math.bib}
\bibliographystyle{alpha}

\end{document}